\documentclass[a4paper, 11pt]{article}
\usepackage[margin = 0.6 in]{geometry}
\usepackage{amsmath, amssymb, amsthm,indentfirst,placeins,float,graphicx,booktabs, verbatim,xcolor,titlesec,mathrsfs,dsfont,subcaption,natbib}

\usepackage[hyperindex,breaklinks]{hyperref}
\hypersetup{linkcolor=blue, citecolor=red, colorlinks=true}

\theoremstyle{plain}
\newtheorem{myth}{Theorem}[section]
\newtheorem{mylm}[myth]{Lemma}
\newtheorem{mypr}[myth]{Proposition}

\newtheorem{myrem}{Remark}[section]

\theoremstyle{definition}
\newtheorem{mydef}[myth]{Definition}

\DeclareMathOperator*{\argmax}{arg\,max}
\DeclareMathOperator*{\argmin}{arg\,min}
\newcommand{\wopt}{W^{y^*,z^*}}
\newcommand{\mc}{\mathcal}
\newcommand{\ul}{\underline}
\newcommand{\ol}{\overline}

\allowdisplaybreaks

\begin{document}

	\title{Optimal Ratcheting of Dividends with Irreversible Reinsurance}
	\author{Tim J. Boonen
	\thanks{Department of Statistics and Actuarial Science, School of Computing and Data Science, The University of Hong Kong, Hong Kong, China. Email: tjboonen@hku.hk}
	\and
	Engel John C. Dela Vega
	\thanks{Department of Statistics and Actuarial Science, School of Computing and Data Science, The University of Hong Kong, Hong Kong, China. Email: ejdv@hku.hk}}
	\date{}
	\maketitle

\begin{abstract}
    This paper considers an insurance company that faces two key constraints: a ratcheting dividend constraint and an irreversible reinsurance constraint. The company allocates part of its reserve to pay dividends to its shareholders while strategically purchasing reinsurance for its claims. The ratcheting dividend constraint ensures that dividend cuts are prohibited at any time.   The irreversible reinsurance constraint ensures that reinsurance contracts cannot be prematurely terminated or sold to external entities. The dividend rate and reinsurance level are modeled as nondecreasing processes, thereby satisfying the constraints. Claims are modeled using a Brownian risk model. The main objective is to maximize the cumulative expected discounted dividend payouts until the time of ruin. The reinsurance and dividend levels are restricted to a finite set. The optimal value function is shown to be the unique viscosity solution of the corresponding Hamilton-Jacobi-Bellman equation. A threshold strategy is constructed and shown to be optimal. Finally, numerical examples are presented to illustrate the optimality conditions and optimal strategies.
\end{abstract}



\maketitle

\section{Introduction}
    The optimal dividend payout problem is a fundamental topic in the field of actuarial science. The foundations of this topic were laid in a seminal paper by \citet{definetti1957}, which explored the optimal way to distribute dividends to the company's shareholders. The key insight is that the optimal dividend payout strategy is the one that maximizes the expected discounted sum of all dividends paid to shareholders up until the company's time of bankruptcy, also called the ruin time. 
	
This optimization problem is critical for financial institutions, including insurance companies and banks, as they must strike a balance between paying dividends to shareholders and holding reserves to cover potential future liabilities. Consequently, \citet{definetti1957} has inspired numerous variations and extensions of the optimal dividend payout problem. An overview of techniques and strategies for solving optimal dividend payout problems is presented in \citet{albrecher2009} and  \citet{avanzi2009}.
	
    The use of control techniques, such as viscosity solutions and the Hamilton–Jacobi–Bellman (HJB) approach, in optimal dividend payout problems is discussed in \citet{azcuebook2014} and \citet{schmidli:controlinsurance}. The case where the dividend rate can be unbounded is also discussed in \citet{schmidli:controlinsurance}. The reserve process is commonly modeled via the classical risk model, sometimes referred to as the Cram\'er-Lundberg model, or the diffusion approximation model. Jump-diffusion processes have also been used to model the reserve process, for example in \citet{belhaj2010}.
	
	The term \emph{ratcheting} is first used in the context of optimal dynamic consumption and investment problems, as discussed in \citet{duesenberry:consumption} and \citet{dybvig1995}, where consumption is not allowed to fall over time and increases proportionately as wealth reaches a new threshold or maximum. The dividend ratcheting constraint is motivated by practitioners’ observations and by the fact that shareholders are generally unhappy about reductions in dividend payments, as discussed in \citet{avanzi2016} for the case of insurance companies. The positive and significant share-price response of insurers to dividend increases is discussed in \citet{akhigbe1993}.
	
	The first time that the dividend ratcheting constraint was introduced in the context of optimal dividend payout problems is in \citet{albrecher2018}, where the dividend rate can be increased only once and must not exceed the drift rate of the original reserve process under the L\'evy risk model and the Brownian risk model. The general case of the ratcheting constraint, where the dividend is allowed to increase more than once during the lifetime of the process, is discussed in \citet{albrecher2020} under the compound Poisson process and in \citet{albrecher2022} under the Brownian risk model. There is more regularity present in the value function under the Brownian risk model; hence, an explicit differential equation was derived and used to identify the candidates for the optimal strategies in the latter work. Some extensions and variations include a habit-formation (or drawdown) constraint with ratcheting dividends as in \citet{angoshtari2019}, 
 a periodic dividend payout with ratcheting as in \citet{song2023} and \citet{sun2024}, and  ratcheting with capital injection as in \citet{wang2024}.    
 
	\citet{chang1996} argue that reinsurance negotiations are expensive, time-consuming, and irreversible, unlike insurance futures and call spreads, which are standardized and exchange-traded. Reinsurance contracts are structured to provide long‑term stability for the insurer, reflecting a mutual understanding between insurer and reinsurer regarding risk-sharing that cannot easily be modified without significant consequences for either party. With irreversible reinsurance, neither party can prematurely terminate the agreement, nor can the insurer lower the reinsurance coverage at any time. This arrangement benefits the insurer by allowing it to hedge an increasing proportion of its risks, thereby enhancing financial stability and the capacity to manage future uncertainties. 
    
    The modeling of irreversible reinsurance is motivated by problems such as the finite-fuel problem in \citet{bank2005}, the irreversible investment problem in \citet{bertola1994}, and the irreversible liquidation problem in \citet{ferrari2021}, all of which employ techniques for singular control problems. In \citet{yan2022}, they formulate the irreversible reinsurance problem using a singular control problem under a mean-reverting risk exposure process. Their objective function consists of a running cost function generated by the risk exposure and the cost of purchasing new reinsurance contracts. The work of \citet{brachetta2021} proposes an optimal reinsurance problem, wherein a reinsurance contract with some fixed cost is purchased and simultaneously the risk retention level is decided at some random time before maturity. The work of \citet{federico2024} studies the minimization of the flow of capital injections while considering purchasing perpetual reinsurance contracts at certain random times.
	
The optimal dividend payout with proportional reinsurance problem under the diffusion approximation, without the ratcheting dividend constraint and the irreversible reinsurance constraint, is first studied in \citet{hojgaard1999}. Excess-of-loss reinsurance policies have also been considered in optimal dividend payout problems, such as in \citet{asmussen2000}. Impulse control has been used in solving optimal dividend payout with proportional reinsurance policies, such as in \citet{wei2010}, where regime-switching was also introduced.
	
    In this work, we study a company that seeks to optimize dividend payouts subject to a ratcheting dividend constraint and an irreversible reinsurance constraint. This extends the work of \citet{albrecher2022} by incorporating reinsurance and introducing the irreversible reinsurance constraint, resulting in a multi-dimensional stochastic control problem. 
    The rate at which the company pays dividends is modeled as a nondecreasing process, while the level of risk retained by the company is modeled as a nonincreasing process. Consequently, the level of reinsurance is modeled as a nondecreasing process.
	
	The discussion begins by modeling the reserve process via the classical Cram\'er-Lundberg model. We assume that the company's premium rate is computed using the expected value principle and that the company's incoming claims are reinsured using \emph{proportional reinsurance}. Besides analytical tractability, proportional reinsurance offers a more straightforward mechanism for risk-sharing. It is also advantageous to the insurer in the sense that the reinsurer is required to participate in all types of claims. 
    
    We use the result of \citet{grandell1977} to obtain a diffusion approximation of the reserve process with proportional reinsurance. The value function exhibits greater regularity under the diffusion approximation of the Cram\'er-Lundberg model, which significantly enhances the analytical tractability of the problem. This allows for an analytical solution to the HJB equation, which serves as a basis for the construction of the optimal strategies. The diffusion approximation also mitigates the complexities associated with sudden claims, since one of its key assumptions is that the expected number of claims is sufficiently large. The dynamics of the controlled reserve process are then introduced, which is essentially the diffusion approximation of the reserve process reduced by the dividend rate process. The goal is then to maximize the expected discounted dividends paid until ruin time.   
	
	The dynamic programming approach is then used to obtain the HJB equation. We start with the case of constant dividend and reinsurance levels, which yields a characteristic equation that motivates the form of the candidate value function. We then discuss the HJB equation for the case where the dividend and reinsurance levels belong to finite sets. We introduce viscosity solutions because the candidate value function is not necessarily twice continuously differentiable. {The form of the value function is derived using two approaches: backward recursion and scale functions.} The optimal strategy is constructed by determining the appropriate threshold levels for the dividend and reinsurance variables. The order in which these levels are changed must be carefully analyzed, which is one of the main differences with the work of \citet{albrecher2022}.

    {The main contributions of this paper are twofold. First, we obtain a recursive formula for the optimal value function when the reinsurance levels and dividend rates belong to finite sets. This restriction allows a clear construction of the optimal reinsurance and dividend strategies, simplifying insurers' decision-making. For instance, insurers can easily select from a finite set of retention levels that align well with their risk appetite when considering additional reinsurance coverage. Moreover, when presenting risk strategies to stakeholders, the use of finite choices creates a clearer narrative. This clarity can significantly improve discussions around risk management strategies, allowing stakeholders to better understand the implications of various decisions.

    Second, the numerical illustrations in Section \ref{section:examples} provide new insights into the dynamics of optimal retention and dividend threshold levels in the context of ratcheting dividend and irreversibility of reinsurance. We find that, under typical market regimes, it is almost always optimal to increase dividends before transferring additional risk to the reinsurer. We also observe that decreasing the retention level requires relatively high reserves. This result contrasts with the findings of \citet{hojgaard1999}, who study an optimal dividend payout problem involving proportional reinsurance without imposing constraints on dividends and reinsurance. In their study, optimal retention levels decrease as the reserve level rises. In our framework, however, the irreversibility constraint requires insurers to increase their reserves before considering an increase in their reinsurance coverage.
    
    A key factor influencing this dynamic is the ratcheting constraint on dividend rates. Increasing dividend rates results in a lower drift under the diffusion approximation model for the reserve process, effectively slowing the growth of the reserve level. This interplay between the two constraints highlights their impact on optimal strategies and offers a fresh and valuable perspective. It also provides actionable insights that may meaningfully influence insurers’ strategic decisions in practice.}
	
	The paper is organized as follows. Section \ref{section:model} discusses the model for the reserve process and the properties of the value function. Section \ref{section:hjb} discusses the main results. The derivation of the optimal strategies is discussed in Section \ref{section:optimalstrat}. Numerical illustrations are presented in Section \ref{section:examples}. Section \ref{section:conc} concludes. Appendix \ref{appendix:a} provides the intermediate lemmas, propositions, and their proofs, while Appendix \ref{appendix: b} contains the proofs of the main results. An alternative derivation of the form of the optimal value function is provided in Appendix \ref{appendix:c}.

\section{Model}\label{section:model}	
	Consider a complete probability space $(\Omega,\mathcal{F},\mathbb{F},\mathbb{P})$, where the filtration $\mathbb{F}:=\{\mathcal{F}_t\}_{t\geq 0}$ is the $\mathbb{P}$-augmentation of the natural filtration $\{\mathcal{F}_t^W\}_{t\geq 0}$ generated by a standard Brownian motion $\{W_t\}_{t\geq 0}$. We consider an insurance company that simultaneously uses part of its reserve to pay dividends to shareholders and strategically reinsures some of its claims.
	
	Suppose that the company's reserve process follows the classical Cramér-Lundberg model given by $R_t=x+pt-\sum_{i=1}^{N_t}Z_i$,	where $x>0$ is the initial reserve, $p>0$ is the premium rate on the risk that the company is insuring, $\{N_t\}_{t\geq 0}$ is a Poisson process modeling claim frequency with intensity parameter $\lambda>0$, and $\{Z_i\}_{i=1,2,\ldots}$ is a series of positive-valued, independent and identically distributed (i.i.d.) random variables with $Z_i$ denoting the size of the $i$th loss. We assume that $\{N_t\}_{t\geq 0}$ and $\{Z_i\}_{i=1,2,\ldots}$ are independent under $\mathbb{P}$. We further suppose that the i.i.d. random variables $Z_1,Z_2,\ldots$ have a common distribution with finite mean $\mu_0$ and finite second moment $\sigma_0^2$.
	
	We assume that the company's premium rate $p$ is calculated using the expected value principle with relative safety loading $\gamma>0$, that is, $p=(1+\gamma)\lambda\mu_0$. We further assume that the claims are reinsured by proportional reinsurance with retention level $A\in[0,1]$. For a claim $Z_i$, the company pays $AZ_i$ and the reinsurer pays $(1-A)Z_i$. Suppose that the reinsurance premium $p^A$ is calculated using the expected value principle with relative safety loading $\theta>0$, that is, $p^A=(1+\theta)(1-A)\lambda\mu_0$. The reserve process with proportional reinsurance, denoted by $R^A$, can then be written as
	\begin{equation*}
		R^A_t=x+(p-p^A)t-A\sum_{i=1}^{N_t}Z_i=x+\left[A(1+\theta)-(\theta-\gamma)\right]\lambda\mu_0t-A\sum_{i=1}^{N_t}Z_i.
	\end{equation*}
    The company can then choose a reinsurance strategy $A=\{A(t)\}_{t\geq 0}$, where $A(t)$ 
 represents the proportion of the risk retained by the insurer, or simply the \emph{retention level}, at time $t$. 
	Following \citet{grandell1977}, the diffusion approximation of the reserve process, denoted by $X=\{X_t\}_{t\geq 0}$ is given by
	\begin{equation}\label{eqn: diff approx}
		X_t=x+\int_0^t\left[\mu A(s)-b\right]ds+\int_0^t\sigma A(s)dW_s,
	\end{equation}
    where $\mu:=\theta\lambda\mu_0>0$, $b:=(\theta-\gamma)\lambda\mu_0$, and $\sigma:=\sqrt{\lambda\sigma_0^2}>0$.

	\begin{myrem}
		The diffusion approximation of the reserve process \eqref{eqn: diff approx} has a form that also works for the following premium principles:
		\begin{itemize}
			\item Standard Deviation Premium Principle: $p=\lambda\mu_0+\gamma\sqrt{\lambda\sigma_0^2}$ and $p^A=\lambda(1-A)\mu_0+\theta(1-A)\sqrt{\lambda\sigma_0^2}$ implies $\mu=\theta\sqrt{\lambda\sigma^2_0}$, $b=(\theta-\gamma)\sqrt{\lambda\sigma_0^2}$, and $\sigma=\sqrt{\lambda\sigma_0^2}$.
			\item Modified Variance Principle: $p=\lambda\mu_0+\gamma\frac{\sigma_0^2}{\mu_0}$ and $p^A=\lambda(1-A)\mu_0+\theta(1-A)\frac{\sigma_0^2}{\mu_0}$ implies $\mu=\theta\frac{\sigma_0^2}{\mu_0}$, $b=(\theta-\gamma)\frac{\sigma_0^2}{\mu_0}$, and $\sigma=\sqrt{\lambda\sigma_0^2}$.
		\end{itemize}
	\end{myrem}
	
	Let $C(t)$ be the rate at which the company pays dividends at time $t$. When applying the policy $\pi:=(A,C)$, let $\{X_t^{\pi}\}_{t\geq 0}$ be the controlled reserve process. The dynamics for $X_t^{\pi}$ can then be written as
	\begin{equation*}
		\begin{cases}
			dX_t^{\pi}=\left[\mu A(t)-b\right]dt+\sigma A(t)dW_t-C(t)dt,\\
			X_0^{\pi}=x.
		\end{cases}
	\end{equation*}
	
	{Suppose that the retention level belongs to a set $\mathscr{A}\subseteq [\underline a,1]$, where $\underline a \in [0,1]$ is the minimum retention level possible and satisfies $\underline a \in \mathscr{A}$, while the dividend rates belong to a set $\mathscr{C}\subseteq[0,\ol{c}]$, where $0\leq \ol{c}\in\mathscr{C}$ is the maximum dividend rate possible. Given an initial reserve $X(0)=x\geq 0$, a maximum initial retention level $a\in \mathscr{A}$, and a minimum initial dividend rate $c\in\mathscr{C}$, the policy $\pi$ is said to be admissible if
	\begin{enumerate}
		\item[(i)] the processes $A:=\{A(t)\}_{t\geq 0}$ and $C:=\{C(t)\}_{t\geq 0}$ are adapted to $\mathbb{F}$,
		\item[(ii)] $A$ is nonincreasing, right-continuous, and $A(t)\in\mathscr{A}$ for all $t\geq 0$,
		\item[(iii)] $C$ is nondecreasing, right-continuous, and $C(t)\in\mathscr{C}$ for all $t\geq 0$.
	\end{enumerate}
	 Denote by $\Pi^{\mathscr{A},\mathscr{C}}_{x,a,c}$ the set of all admissible policies. 
    \begin{myrem}\label{remark on inclusion}
        A maximum initial retention level $a\in \mathscr A$ and a minimum initial dividend rate $c\in\mathscr C$ mean that $\underline a\leq A(0) \leq a$ and $c\leq C(0) \leq \ol c$. This further implies that given $a_1,a_2\in\mathscr A$ with $a_1\geq a_2$, we have $\Pi^{\mathscr{A},\mathscr{C}}_{x,a_2,c}\subseteq\Pi^{\mathscr{A},\mathscr{C}}_{x,a_1,c}$ for any $x\geq 0$ and $c\in\mathscr C$. In a similar way, given $c_1,c_2\in\mathscr C$ with $c_1\leq c_2$, we have $\Pi^{\mathscr{A},\mathscr{C}}_{x,a,c_2}\subseteq\Pi^{\mathscr{A},\mathscr{C}}_{x,a,c_1}$ for any $x\geq 0$ and $a\in\mathscr A$. It is important to note that this inclusion property will not hold if the definition of admissible policies requires $A(0) = a$ and $C(0) = c$. For clarity, we also mention that $a=\argmax \mathscr A$ and $c=\argmin \mathscr C$ do not necessarily hold.
    \end{myrem}
    \begin{myrem}
        The maximum dividend rate $\overline c$ establishes an upper bound on the dividend rates. On the other hand, the minimum retention level $\underline a$ ensures that retention levels do not fall below this threshold. Consequently, if $\underline a>0$, then ``full reinsurance" is not possible regardless of the reserve level.
    \end{myrem}}

	 Given $\pi=(A,C)\in\Pi^{\mathscr{A},\mathscr{C}}_{x,a,c}$, the value function is given by
	 \begin{equation*}
	 	J(x;A,C)=\mathbb{E}\left[\int_0^{\tau_{\pi}}e^{-qs}C(s)ds\right],
	 \end{equation*}
	 where $q>0$ is the discount factor and $\tau_{\pi}:=\inf\{t\geq0:X_t^{\pi}<0\}$ is the ruin time. For any initial reserve $x\geq 0$, initial retention level $a$, and initial dividend rate $c$, the goal is to find the optimal value function defined as
	 \begin{equation}\label{value function orig}
	 	V^{\mathscr{A},\mathscr{C}}(x,a,c)=\sup_{\pi\in\Pi^{\mathscr{A},\mathscr{C}}_{x,a,c}}J(x;A,C).
	 \end{equation} To simplify the notation, we write $V:=V^{\mathscr{A},\mathscr{C}}$ for the general sets $\mathscr{A}$ and $\mathscr{C}$.
     
	 \begin{myrem}\label{remark on no constraint case}
	 	The optimal dividend payout with proportional reinsurance problem under the diffusion approximation without the ratcheting dividend constraint and the irreversible reinsurance constraint is first studied in \citet{hojgaard1999}. Their model setup is similar to this work without the constraints on the dividends and the reinsurance and without the constant term in the drift coefficient of the reserve process. Their problem is considered to be one-dimensional without these constraints. The additional constant term in the drift is a trivial extension of their work, and the same results should still hold. Hence, if we denote by $V_{NC}(x)$ the optimal value function without constraints, then it is clear that $V_{NC}(x)\geq V(x,a,c)$ for all $x\geq 0$, $a\in\mathscr{A}$, and $c\in\mathscr{C}$. By Theorem 2.3 of \citet{hojgaard1999}, $V_{NC}$ is increasing, concave, twice continuously differentiable with $V_{NC}(0)=0$ and $\lim_{x\to\infty}V_{NC}(x)=\frac{\ol c}{q}$. Thus, for any $0\leq x_1\leq x_2$, $V_{NC}$ is Lipschitz, that is, it satisfies
	 	\begin{equation}\label{no constraint Lipschitz}
	 		0\leq V_{NC}(x_2)-V_{NC}(x_1)\leq L_0(x_2-x_1),
	 	\end{equation}
        where $L_0=V'_{NC}(0)$.
	 \end{myrem}
	 
	The above remark on the relationship between the optimal value functions with constraints and no constraints is useful in proving some properties of the optimal value function $V$ considered in this paper. In the remainder of this section, we state some properties of $V$. {Moreover, for the rest of the paper, all of the proofs are detailed in the Appendix. The following proposition states that $V$ satisfies boundedness and monotonicity.} 
	 
	 \begin{mypr}\label{bound and monotone}
	 	The optimal value function $V(x,a,c)$ is bounded above by $\frac{\ol{c}}{q}$, nondecreasing in $x$ and $a$, and nonincreasing in $c$.
	 \end{mypr}
	 
	 The optimal value function $V$ cannot exceed the upper limit $\frac{\ol c}{q}$ given any reserve level, and it approaches $\frac{\ol c}{q}$ as the reserve levels become high enough. Moreover, higher reserve and retention levels imply that $V$ reaches the upper limit faster, while higher dividend levels imply that $V$ reaches the upper limit slower.

	 The next proposition states that $V$ satisfies a Lipschitz property with respect to all of its variables. The proof requires the results in \citet{claisse2016} for conditional expectations involving stopping times.
	 
	 \begin{mypr}\label{Lipschitz property}
	 	There exists a constant $L>0$ such that
	 	\begin{equation*}
	 		0\leq V(x_2,a_1,c_1)-V(x_1,a_2,c_2)\leq L\left[(x_2-x_1)+(a_1-a_2)+(c_2-c_1)\right]
	 	\end{equation*}	 	
       for all $0\leq x_1\leq x_2$, $a_1,a_2\in\mathscr{A}$ with $a_2\leq a_1$, and $c_1,c_2\in\mathscr{C}$ with $c_1\leq c_2$. 
	 \end{mypr}
	 
	 One way to interpret this property is that small or large perturbations in the reserve level, retention rate, or dividend rate do not lead to disproportionately large changes in the value of $V$. More precisely, the Lipschitz result implies that the change in $V$ is bounded by a term that is proportional to the change in input values. Moreover, if $V$ has bounded partial derivatives in $x$, $a$, and $c$, then the Lipschitz property is immediately satisfied. We will see in the next sections that $V$ may not be smooth, that is, the partial derivatives may not exist.
	 
	 The following lemma states that the optimal value function $V$ satisfies the dynamic programming principle, which is essential to prove that $V$ is a viscosity solution of the HJB equation, which is the subject of the next section. {The proof is similar to that of \citet[Lemma 1.2]{azcuebook2014}.}
	 
	 \begin{mylm}\label{dpp}
	 	Given any stopping time $\tilde{\tau}$, we have
	 	\begin{equation*}
	 		V(x,a,c)=\sup_{(A,C)\in\Pi^{\mathscr{A},\mathscr{C}}_{x,a,c}}\mathbb{E}\left[\int_0^{\tau_{\pi}\wedge \tilde{\tau}}e^{-qs}C(s)ds+e^{-q(\tau_{\pi}\wedge \tilde{\tau})}V(X^{\pi}_{\tau_{\pi}\wedge \tilde{\tau}},A(\tau_{\pi}\wedge \tilde{\tau}),C(\tau_{\pi}\wedge \tilde{\tau}))\right].
	 	\end{equation*}
	 \end{mylm}

\section{Main Results}\label{section:hjb}
	 
	 In this section, we present the HJB equations for two different cases, each characterized by the form of the reinsurance set $\mathscr{A}$ and the dividend set $\mathscr{C}$: (1) \textbf{singleton sets}, where the retention level and the dividend rate are fixed, and (2) \textbf{finite sets}, where the retention level and the dividend rate can take on any value from a discrete set of choices. For the finite set case, we will prove that the optimal value function is the unique viscosity solution to its corresponding HJB equation.
	
	 \subsection{Singleton Case}
	 Consider the case $\mathscr{A}=\{a\}$ and $\mathscr{C}=\{c\}$ (i.e., the singleton set case). In this case, the unique admissible strategy consists of having a constant retention level $a$ retained by the insurer and paying a constant dividend rate $c$ up to the ruin time. The value function, denoted by $V^{\{a\},\{c\}}(x,a,c)$, is the unique solution of the second-order linear ordinary differential equation:
	 \begin{equation}\label{HJB base}
	 	\mathcal{L}^{a,c}\left(V^{\{a\},\{c\}}\right):=\frac{1}{2}\sigma^2a^2V^{\{a\},\{c\}}_{xx}+(\mu a-b-c)V^{\{a\},\{c\}}_x-qV^{\{a\},\{c\}}+c=0
	 \end{equation}
	 with boundary conditions $V^{\{a\},\{c\}}(0,a,c)=0$ and $\lim_{x\to\infty}V^{\{a\},\{c\}}(x,a,c)=\frac{c}{q}$. The solutions of (\ref{HJB base}) are of the form $K_1e^{\theta_1(a,c)x}+K_2e^{\theta_2(a,c)x}+\frac{c}{q}$, where $K_1,K_2\in\mathbb{R}$, and $\theta_1(a,c)$ and $\theta_2(a,c)$ are defined as $\theta_1(a,c):=\frac{b+c-\mu a+\sqrt{(b+c-\mu a)^2+2q\sigma^2 a^2}}{\sigma^2a^2}$ and $\theta_2(a,c):=\frac{b+c-\mu a-\sqrt{(b+c-\mu a)^2+2q\sigma^2 a^2}}{\sigma^2a^2}$. The properties of the functions $\theta_1(a,c)$ and $\theta_2(a,c)$ and their partial derivatives are stated in Proposition \ref{claim on theta 1 and 2}.

	 The solutions of \eqref{HJB base} with boundary condition $V^{\{a\},\{c\}}(0)=0$ are of the form
	 \begin{equation}\label{ode solution base}
	 	K\left[e^{\theta_1(a,c)x}-e^{\theta_2(a,c)x}\right]+\frac{c}{q}\left[1-e^{\theta_2(a,c)x}\right],
	 \end{equation}
	 where $K\in\mathbb{R}$. The unique solution of \eqref{HJB base} satisfying the boundary conditions $V^{\{a\},\{c\}}(0,a,c)=0$ and $\lim_{x\to\infty}V^{\{a\},\{c\}}(x,a,c)=\frac{c}{q}$ corresponds to $K=0$. Hence, we have  $V^{\{a\},\{c\}}(x,a,c)=\frac{c}{q}\left[1-e^{\theta_2(a,c)x}\right]$. It is clear that $V^{\{a\},\{c\}}(x,a,c)$ is nondecreasing and concave in $x$.
	 
	 \begin{myrem}\label{remark on asymptote V}
	 	We have $V(x,a,c)\geq V^{\{a\},\{c\}}(x,a,{c})=\frac{{c}}{q}\left[1-e^{\theta_2(a,{c})x}\right]$. Together with Proposition \ref{bound and monotone}, we obtain $\lim_{x\to\infty}V(x,a,c)=\frac{{c}}{q}$.
	 \end{myrem}
	
	 \subsection{Finite Set Case}
	 We now consider the finite set case; that is, $\mathscr{A}:=\{a_1,a_2,\ldots,a_m\}$ with $\underline a= a_m<a_{m-1}<\cdots <a_1\leq 1$ and $\mathscr{C}:=\{c_1,c_2,\ldots,c_n\}$ with $0\leq c_1<c_2<\cdots<c_n=\ol{c}$. By definition, $V^{\mathscr{A},\mathscr{C}}(x,a_i,c_j)=V(x,a_i,c_j)$.
	 
	 For $1\leq i\leq m$ and $1\leq j\leq n$, we simplify the notation as follows:
	 \begin{equation}\label{value discrete}
	 	V^{a_i,c_j}(x):=V^{\mathscr{A},\mathscr{C}}(x,a_i,c_j).
	 \end{equation}
	 Using Proposition \ref{bound and monotone}, we have $V^{a_i,c_j}(x)\geq V^{a_{i+1},c_j}(x)\geq \cdots\geq V^{a_m,c_j}(x)$ for a fixed $1\leq j\leq n-1$, and $V^{a_i,c_j}(x)\geq V^{a_{i},c_{j+1}}(x)\geq \cdots \geq V^{a_i,c_n}(x)$ for a fixed $1\leq i\leq m-1$. Moreover, we set $V^{a_i,c_n}(x)=0$ and $V^{a_m,c_j}=0$ for $i<n$ and $j<m$. The HJB equation associated to (\ref{value discrete}) is given by
	\begin{equation}\label{hjb main}
		\max\{\mathcal{L}^{a_i,c_j}(V^{a_i,c_j})(x),V^{a_{i+1},c_{j}}(x)-V^{a_i,c_j}(x),V^{a_i,c_{j+1}}(x)-V^{a_i,c_j}(x)\}=0,
	\end{equation}
	for $x\geq 0$, $1\leq i\leq m-1$, and $1\leq j\leq n-1$.
	
	We now define viscosity subsolutions and supersolutions for the HJB equation \eqref{hjb main}.
	
	{\begin{mydef}
		A (locally Lipschitz) function $u:[0,\infty)\to\mathbb{R}$ is a viscosity supersolution (subsolution, resp.) of (\ref{hjb main}) at $x\in(0,\infty)$ if any twice continuously differentiable function $\varphi:[0,\infty)\to\mathbb{R}$ 
            with $\varphi(x)={u}(x)$, such that ${u}-\varphi$ reaches a minimum (maximum, resp.) at $x$, satisfies
			\begin{equation}\label{hjb supersolution}
				\max\{\mathcal{L}^{a_i,c_j}(\varphi)(x),V^{a_{i+1},c_{j}}(x)-\varphi(x),V^{a_i,c_{j+1}}(x)-\varphi(x)\}\leq 0\quad (\geq 0, \mbox{ resp}.).
			\end{equation}
        Moreover, $u$ is a viscosity solution if it is both a viscosity supersolution and a viscosity subsolution. 
	\end{mydef}}
    
	We now state the first main result. The following theorem states that $V^{a_i,c_j}$ is the unique viscosity solution of \eqref{hjb main} and satisfies the boundary conditions. 
	
	\begin{myth}\label{verification discrete}
		The optimal value function $V^{a_i,c_j}(x)$ for $i=1,\ldots,m-1$ and $j=1,\ldots, n-1$ is the unique viscosity solution of the associated HJB equation (\ref{hjb main}) with the boundary conditions $V^{a_i,c_j}(0)=0$ and $\lim_{x\to\infty}V^{a_i,c_j}(x)=\frac{\overline c}{q}$.
	\end{myth}

    Since $V^{a_i,c_j}(x)$ is the viscosity solution of \eqref{hjb main} via Theorem \ref{verification discrete}, we can infer that there are values of $x$ such that (i) $\mathcal{L}^{a_i,c_j}(V^{a_i,c_j})=0$, (ii) $V^{a_i,c_j}(x)=V^{a_{i+1},c_j}(x)$, and (iii) $V^{a_i,c_j}(x)=V^{a_{i},c_{j+1}}(x)$. Hence, we can partition $(0,\infty)$ in two ways. The first way is fixing $j$ (i.e., the dividend level is fixed). For $i<m$, $(0,\infty)$ can be partitioned in such a way that the optimal strategy is to retain $a_i$ of the incoming claims when the current reserve is in the open set $\{x:V^{a_i,c_j}(x)>V^{a_{i+1},c_j}(x)\}$ and to decrease the risk exposure to $a_{i+1}$ when the current reserve is in the closed set $\{x:V^{a_i,c_j}(x)=V^{a_{i+1},c_j}(x)\}$. The second way is fixing $i$ (i.e., fixing the retention level). Similarly, for $j<n$, the optimal strategy is to pay dividends at rate $c_j$ of the incoming claims when the current reserve is in the open set $\{x:V^{a_i,c_j}(x)>V^{a_{i},c_{j+1}}(x)\}$ and to increase the dividend rate to $c_{j+1}$ when the current reserve is in the closed set $\{x:V^{a_i,c_j}(x)=V^{a_{i},c_{j+1}}(x)\}$. 

    Define the functions $y:\mathscr{A}\times\mathscr{C}\to[0,\infty]$ and $z:\mathscr{A}\times\mathscr{C}\to[0,\infty]$ such that $y(a_m,c_j)=+\infty$ for $1\leq j\leq n$ and $z(a_i,c_n)=+\infty$ for $1\leq i\leq m$. For $1\leq i\leq m$ and $1\leq j\leq n$, we simplify the notation as follows
	\begin{equation*}
		y_{i,j}:=y(a_i,c_j)\quad\mbox{and}\quad z_{i,j}:=z(a_i,c_j).
	\end{equation*}
    We call the values $y_{i,j}$ and $z_{i,j}$ the retention threshold and dividend threshold, respectively, at retention level $a_i$ and dividend rate $c_j$. The functions $y$ and $z$ are called the retention threshold function and the dividend threshold function, respectively.

	The next theorem states the form of the optimal value function.

    \begin{myth}\label{y and z optimal implies verification}
        Define the function $W^{y,z}(x,a_i,c_j)$ such that it satisfies the following recursive formula:
		\small\begin{equation*}
        \begin{aligned}
            W^{y,z}(x,a_m,c_n)&=\frac{c_n}{q}\left[1-e^{\theta_2(a_m,c_n)x}\right],\\
			W^{y,z}(x,a_i,c_j)&=\begin{cases}
				\frac{c_j}{q}\left[1-e^{\theta_2(a_i,c_j)x}\right]
                +k^{y,z}(a_i,c_j)\left[e^{\theta_1(a_i,c_j)x}-e^{\theta_2(a_i,c_j)x}\right]&\mbox{if $x<y_{i,j}\wedge z_{i,j}$},\\
				W^{y,z}(x,a_{i+1},c_j)\mathbf{1}_{\{y_{i,j}<z_{i,j}\}}
                +W^{y,z}(x,a_i,c_{j+1})\mathbf{1}_{\{y_{i,j}>z_{i,j}\}}\\
				+W^{y,z}(x,a_{i+1},c_{j+1})\mathbf{1}_{\{y_{i,j}=z_{i,j}\}}&\mbox{otherwise},
			\end{cases}
        \end{aligned}
		\end{equation*}\normalsize
		for $i<n$ or $j<m$, where
		\begin{equation*}
			k^{y,z}(a_i,c_j):=\begin{cases}
				\frac{W^{y,z}(y_{i,j},a_{i+1},c_j)-\frac{c_j}{q}\left[1-e^{\theta_2(a_i,c_j)y_{i,j}}\right]}{e^{\theta_1(a_i,c_j)y_{i,j}}-e^{\theta_2(a_i,c_j)y_{i,j}}}&\mbox{if $y_{i,j}<z_{i,j}$},\\
				\frac{W^{y,z}(z_{i,j},a_{i},c_{j+1})-\frac{c_j}{q}\left[1-e^{\theta_2(a_i,c_j)z_{i,j}}\right]}{e^{\theta_1(a_i,c_j)z_{i,j}}-e^{\theta_2(a_i,c_j)z_{i,j}}}&\mbox{if $y_{i,j}>z_{i,j}$},\\
				\frac{W^{y,z}(y_{i,j},a_{i+1},c_{j+1})-\frac{c_j}{q}\left[1-e^{\theta_2(a_i,c_j)y_{i,j}}\right]}{e^{\theta_1(a_i,c_j)y_{i,j}}-e^{\theta_2(a_i,c_j)y_{i,j}}}&\mbox{if $y_{i,j}=z_{i,j}$}.
			\end{cases}
		\end{equation*}
		Suppose $y^*$ and $z^*$ are the optimal threshold functions in the sense that $k^{y^*,z^*}(a_i,c_j)$ is the maximum for any $i<m$ or $j<n$. Then for $i=1,\ldots,m$ and $j=1,\ldots, n$, we have that $W^{y^*,z^*}(x,a_i,c_j)$ is the optimal value function $V^{a_i,c_j}(x)$ in (\ref{value function orig}).
	\end{myth}

    {\begin{myrem}
        An alternative proof of the value function is given in Appendix \ref{appendix:c}. The derivation involves scale functions in the theory of L\'{e}vy fluctuations. For more details, we refer the interested readers to  \cite{kyprianou2014book}.
    \end{myrem}
    
    \begin{myrem}
        If $i=m$ and $n=j$, then the corresponding problem is reduced to the singleton case. This is because we have reached the maximum possible dividend rate and the minimum possible retention level. Hence, the value function is then concave in $x$. If $i<m$ or $n<j$, the concavity of the value function is not guaranteed. Figure \ref{fig:plot-1-v2} in Section \ref{section:examples} shows an example of a value function that is not necessarily concave.
    \end{myrem}}

    \section{Derivation of the Optimal Strategy}\label{section:optimalstrat}

    {In the previous section, we presented the form of the optimal value function. However, the optimal threshold functions $y^*$ and $z^*$ in Theorem \ref{y and z optimal implies verification} have not yet been specified. We can immediately observe from Theorem \ref{y and z optimal implies verification} that the optimal reinsurance and dividend strategies are of a stationary threshold type. This means that the strategy depends on the current levels of reserve, retention, and dividend. In this section, we will discuss how the form of the value function is derived and how the optimal threshold strategies $y^*$ and $z^*$ are constructed. 
	}

%
Consider the reinsurance-dividend strategies wherein the corresponding value function $\upsilon^{a_i,c_j}(x)$ satisfies the following:
	\begin{enumerate}
		\item[(i)] $\upsilon^{a_m,c_j}(x)=V^{a_m,c_j}(x)$ for $1\leq j\leq n$ and  $\upsilon^{a_i,c_n}(x)=V^{a_i,c_n}(x)$ for $1\leq i\leq m$.
		\item[(ii)] For $i<m$, $\mathcal{L}^{a_i,c_n}(\upsilon^{a_i,c_n}(x))=0$ if $x\in[0,y_{i,n})$ and $\upsilon^{a_i,c_n}(x)=\upsilon^{a_{i+1},c_n}(x)$ if $x\in[y_{i,n},\infty)$.
		\item[(iii)] For $j<n$, $\mathcal{L}^{a_m,c_j}(\upsilon^{a_m,c_j}(x))=0$ if $x\in[0,z_{m,j})$ and $\upsilon^{a_m,c_j}(x)=\upsilon^{a_{m},c_{j+1}}(x)$ if $x\in[z_{m,j},\infty)$.
		\item[(iv)] For $i<m$ and $j<n$, $\mathcal{L}^{a_i,c_j}(\upsilon^{a_i,c_j}(x))=0$ if $x\in[0,y_{i,j}\wedge z_{i,j})$, $\upsilon^{a_i,c_j}(x)=\upsilon^{a_{i+1},c_j}(x)$ if $x\in[y_{i,j},\infty)$, and $\upsilon^{a_i,c_j}(x)=\upsilon^{a_i,c_{j+1}}(x)$ if $x\in[z_{i,j},\infty)$.
	\end{enumerate}

    {Denote the threshold strategy as
	\begin{equation}\label{threshold strategy}
		\pi^{y,z}:=\{(A_{x,a_i,c_j},C_{x,a_i,c_j})\}_{(x,a_i,c_j)\in[0,\infty)\times\mathscr{A}\times\mathscr{C}},
	\end{equation} 
	where $(A_{x,a_i,c_j},C_{x,a_i,c_j})\in\Pi^{\mathscr{A},\mathscr{C}}_{x,a_i,c_j}$. The threshold strategy $\pi^{y,z}$ is defined via backward recursion as follows:}
	\begin{enumerate}
		\item[(i)] If $i=m$ and $j=n$, retain $a_m$ of the incoming claims and pay dividends with rate $c_n$ until ruin time, that is, $(A_{x,a_m,c_n},C_{x,a_m,c_n})(t)=(a_m,c_n)$.
		\item[(ii)] If $i=m$, $j<n$, and $x\in[z_{m,j},\infty)$, follow $(A_{x,a_m,c_l},C_{x,a_m,c_l})\in\Pi^{\mathscr{A},\mathscr{C}}_{x,a_m,c_l}$, where 
		\begin{equation*}
			l=\begin{cases}
				n&\mbox{if $x\geq z_{m,r}$, $r=j+1,\ldots,n-1$},\\
				\min\{r\in\{j+1,\ldots,n-1\}:x<z_{m,r}\}&\mbox{otherwise}.
			\end{cases}
		\end{equation*}
		\item[(iii)] If $j=n$, $i<m$, and $x\in[y_{i,n},\infty)$, follow $(A_{x,a_k,c_n},C_{x,a_k,c_n})\in\Pi^{\mathscr{A},\mathscr{C}}_{x,a_k,c_n}$, where 
		\begin{equation*}
			k=\begin{cases}
				m&\mbox{if $x\geq y_{s,n}$, $s=i+1,\ldots,m-1$},\\
				\min\{s\in\{i+1,\ldots,m-1\}:x<y_{s,n}\}&\mbox{otherwise}.
			\end{cases}
		\end{equation*}
		\item[(iv)] If $i<m$, $j<n$, and $x\in [y_{i,j},z_{i,j})$, follow $(A_{x,a_k,c_j},C_{x,a_k,c_j})\in\Pi^{\mathscr{A},\mathscr{C}}_{x,a_k,c_j}$, where 
		\begin{equation*}
			k=\begin{cases}
				m&\mbox{if $z_{s,j}>x\geq y_{s,j}$, $s=i+1,\ldots,m-1$},\\
				\min\{s\in\{i+1,\ldots,m-1\}:z_{s,j}<y_{s,j}\}&\mbox{otherwise}.
			\end{cases}
		\end{equation*}
		\item[(v)] If $i<m$, $j<n$, and $x\in [z_{i,j},y_{i,j})$, follow $(A_{x,a_i,c_l},C_{x,a_i,c_l})\in\Pi^{\mathscr{A},\mathscr{C}}_{x,a_i,c_l}$, where 
		\begin{equation*}
			l=\begin{cases}
				n&\mbox{if $y_{i,r}>x\geq z_{i,r}$, $r=i+1,\ldots,n-1$},\\
				\min\{r\in\{j+1,\ldots,n-1\}:y_{i,r}<z_{i,r}\}&\mbox{otherwise}.
			\end{cases}
		\end{equation*}
		\item[(vi)] If $i<m$, $j<n$, and $x\in [y_{i,j}\vee z_{i,j},\infty)$, follow $(A_{x,a_k,c_l},C_{x,a_k,c_l})\in\Pi^{\mathscr{A},\mathscr{C}}_{x,a_k,c_l}$, where 
		\begin{equation*}
			k=\begin{cases}
				m&\mbox{if $x\geq y_{s,j}$, $s=i+1,\ldots,m-1$},\\
				\min\{s\in\{i+1,\ldots,m-1\}:x<y_{s,j}\}& \mbox{otherwise},
			\end{cases}
		\end{equation*}
		and
		\begin{equation*}
			l=\begin{cases}
				n&\mbox{if $x\geq z_{i,r}$, $r=j+1,\ldots,n-1$},\\
				\min\{r\in\{j+1,\ldots,n-1\}:x<z_{i,r}\}&\mbox{otherwise}.
			\end{cases}
		\end{equation*}
		\item[(vii)] If $i<m$, $j<n$, and $x\in [0,y_{i,j}\wedge z_{i,j})$,  retain $a_i$ of the incoming claims and pay dividends with rate $c_j$ until ruin time $\tau_{\pi}$ or until before the current reserve reaches $y_{i,j}\wedge z_{i,j}$, whichever comes first. {If the current reserves reach $y_{i,j}$ before ruin time, follow $(A_{x,a_{i+1},c_j},C_{x,a_{i+1},c_j})\in \Pi^{\mathscr{A},\mathscr{C}}_{x,a_{i+1},c_j}$. If the reserves reach $z_{i,j}$ before ruin time, follow $(A_{x,a_{},c_{j+1}},C_{x,a_{i},c_{j+1}})\in \Pi^{\mathscr{A},\mathscr{C}}_{x,a_{i},c_{j+1}}$. If $y_{i,j}=z_{i,j}$ is reached before ruin time, follow $(A_{x,a_{i+1},c_{j+1}},C_{x,a_{i+1},c_{j+1}})\in \Pi^{\mathscr{A},\mathscr{C}}_{x,a_{i+1},c_{j+1}}$.
        That is,
        \begin{equation*}
            (A_{x,a_i,c_j},C_{x,a_i,c_j})(t)=
            \begin{cases}
                (a_i,c_j)& \mbox{if $t<(\tau_a\wedge \tau_c)<\tau_{\pi}$ or $t<\tau_{\pi}<(\tau_a\wedge \tau_c)$,}\\
                (a_{i+1},c_j)& \mbox{if $\tau_a<t<(\tau_{\pi}\wedge \tau_c)$,}\\
                (a_{i},c_{j+1})& \mbox{if $\tau_c<t<(\tau_{\pi}\wedge \tau_a)$,}\\
                (a_{i+1},c_{j+1})& \mbox{if $(\tau_a\vee \tau_c)<t<\tau_{\pi}$,}
            \end{cases}
        \end{equation*}
        where $\tau_a$ and $\tau_c$ are the first times the reserve level hits $y_{i,j}$ and $z_{i,j}$, respectively.}
	\end{enumerate}
    The expected payoff of the strategy $\pi^{y,z}$ is defined as $J(x;A_{x,a_i,c_j},C_{x,a_i,c_j})$ and is equal to $W^{y,z}(x,a_i,c_j)$ defined in Theorem \ref{y and z optimal implies verification}. Since the strategy is to retain $a_i$ of the incoming claims and pay dividends with rate $c_j$ whenever $x\in[0,y_{i,j}\wedge z_{i,j})$, then $\mathcal{L}^{a_i,c_j}(W^{y,z})(x,a_i,c_j)=0$ for $x\in[0,y_{i,j}\wedge z_{i,j})$. Moreover, $W^{y,z}(0,a_i,c_j)=0$ since ruin is immediate as soon as the reserve level hits $0$. Lastly, by definition, $W^{y,z}(x,a_i,c_j)=W^{y,z}(x,a_{i+1},c_j)$ if $x\geq y_{i,j}$ and $W^{y,z}(x,a_i,c_j)=W^{y,z}(x,a_{i},c_{j+1})$ if $x\geq z_{i,j}$. The formula for $W^{y,z}$ then follows from \eqref{ode solution base}.

We now seek to maximize the expected payoff $W^{y,z}(x,a_i,c_j)$ over all possible threshold functions $y$ and $z$. Recall from Theorem \ref{y and z optimal implies verification} that $y^*$ and $z^*$ are the optimal threshold functions in the sense that $k^{y^*,z^*}(a_i,c_j)$ attains the maximum for any $i<m$ and $j<n$. Since the function $W^{y,z}(x,a_m,c_n)$ is known, we can solve this optimization problem using backward recursion.
	
	We look for the optimal threshold strategy, which can be viewed as a sequence of $(m-1)\times(n-1)$ one-dimensional optimization problems. These optimization problems consist in maximizing $k^{y,z}(a_i,c_j)$ for $i=m-1,\ldots,1$ and $j=n-1,\ldots,1$. If $W^{y^*,z^*}(x,a_k,c_l)$, $y^*_{k,l}$, and $z^*_{k,l}$ are known for $k=i+1,\ldots,m-1$ and $l=j+1,\ldots,n-1$, then using the recursive formula for $W^{y,z}$, we can solve for $W^{y^*,z^*}(x,a_{i+1},c_j)$ and $W^{y^*,z^*}(x,a_i,c_{j+1})$. Consequently, we can solve for $W^{y^*,z^*}(x,a_i,c_j)$, $y^*_{i,j}$, and $z^*_{i,j}$.
	
	Define the continuous functions $G^{\mathcal{A}}_{i,j}:[0,\infty)\to\mathbb{R}$, $G^{\mathcal{C}}_{i,j}:[0,\infty)\to\mathbb{R}$, and $G^{\mathcal{E}}_{i,j}:[0,\infty)\to\mathbb{R}$ as
	\begin{equation}\label{G^A function}
		G^{\mathcal{A}}_{i,j}(x):=\begin{cases}
			\frac{W^{y^*,z^*}(x,a_{i+1},c_j)-\frac{c_j}{q}\left[1-e^{\theta_2(a_i,c_j)x}\right]}{e^{\theta_1(a_i,c_j)x}-e^{\theta_2(a_i,c_j)x}}&\mbox{if $x>0$},\\
			\frac{\partial_xW^{y^*,z^*}(0,a_{i+1},c_j)+\frac{c_j}{q}\theta_2(a_i,c_j)}{\theta_1(a_i,c_j)-\theta_2(a_i,c_j)}&\mbox{if $x=0$},
		\end{cases}
	\end{equation}
	\begin{equation*}
		G^{\mathcal{C}}_{i,j}(x):=\begin{cases}
			\frac{W^{y^*,z^*}(x,a_{i},c_{j+1})-\frac{c_j}{q}\left[1-e^{\theta_2(a_i,c_j)x}\right]}{e^{\theta_1(a_i,c_j)x}-e^{\theta_2(a_i,c_j)x}}&\mbox{if $x>0$},\\
			\frac{\partial_xW^{y^*,z^*}(0,a_{i},c_{j+1})+\frac{c_j}{q}\theta_2(a_i,c_j)}{\theta_1(a_i,c_j)-\theta_2(a_i,c_j)}&\mbox{if $x=0$},
		\end{cases}
	\end{equation*}
	and
	\begin{equation*}
		G^{\mathcal{E}}_{i,j}(x):=\begin{cases}
			\frac{W^{y^*,z^*}(x,a_{i+1},c_{j+1})-\frac{c_j}{q}\left[1-e^{\theta_2(a_i,c_j)x}\right]}{e^{\theta_1(a_i,c_j)x}-e^{\theta_2(a_i,c_j)x}}&\mbox{if $x>0$},\\
			\frac{\partial_xW^{y^*,z^*}(0,a_{i+1},c_{j+1})+\frac{c_j}{q}\theta_2(a_i,c_j)}{\theta_1(a_i,c_j)-\theta_2(a_i,c_j)}&\mbox{if $x=0$}.
		\end{cases}
	\end{equation*}
    For $k=\mc A,\mc C,\mc E$, we introduce the following notation:
	\begin{equation*}
		g^{k}_{i,j}:=
		\min\left[\argmax_{x\in[0,\infty)}G^{k}_{i,j}(x)\right].
	\end{equation*}
    By Lemma \ref{G has maxima and min exists}, the maximizer of $G^{k}_{i,j}$ exists in $[0,\infty)$ and $g^{k}_{i,j}$ also exists for $k=\mc A,\mc C,\mc E$. The optimal threshold functions $y^*_{i,j}$ and $z^*_{i,j}$ are then given by:\\
		\textbf{Case 1:} If $\max_{x\in[0,\infty)}G^{\mathcal{A}}_{i,j}(x)>\max_{x\in[0,\infty)}G^{\mathcal{C}}_{i,j}(x)\vee\max_{x\in[0,\infty)}G^{\mathcal{E}}_{i,j}(x)$, then
		\begin{equation*}
			y^*_{i,j}=g^{\mathcal{A}}_{i,j}\quad\mbox{and}\quad z^*_{i,j}=+\infty.
		\end{equation*}
		\textbf{Case 2:} If $\max_{x\in[0,\infty)}G^{\mathcal{C}}_{i,j}(x)>\max_{x\in[0,\infty)}G^{\mathcal{A}}_{i,j}(x)\vee\max_{x\in[0,\infty)}G^{\mathcal{E}}_{i,j}(x)$, then
		\begin{equation*}
			y^*_{i,j}=+\infty\quad\mbox{and}\quad z^*_{i,j}=g^{\mathcal{C}}_{i,j}.
		\end{equation*}
		\textbf{Case 3:} If $\max_{x\in[0,\infty)}G^{\mathcal{E}}_{i,j}(x)>\max_{x\in[0,\infty)}G^{\mathcal{A}}_{i,j}(x)\vee\max_{x\in[0,\infty)}G^{\mathcal{C}}_{i,j}(x)$, then
		\begin{equation*}
			y^*_{i,j}=z^*_{i,j}=g^{\mathcal{E}}_{i,j}.
		\end{equation*}
		\textbf{Case 4:} If $\max_{x\in[0,\infty)}G^{\mathcal{A}}_{i,j}(x)=\max_{x\in[0,\infty)}G^{\mathcal{C}}_{i,j}(x)>\max_{x\in[0,\infty)}G^{\mathcal{E}}_{i,j}(x)$, then either
		\begin{equation*}
			y^*_{i,j}=\begin{cases}
				g^{\mathcal{A}}_{i,j}&\mbox{if $g^{\mathcal{A}}_{i,j}\leq g^{\mathcal{C}}_{i,j}$},\\
				+\infty&\mbox{otherwise},
			\end{cases}
			\quad\mbox{and}\quad
			z^*_{i,j}=\begin{cases}
				g^{\mathcal{C}}_{i,j}&\mbox{if $g^{\mathcal{C}}_{i,j}< g^{\mathcal{A}}_{i,j}$},\\
				+\infty&\mbox{otherwise},
			\end{cases}
		\end{equation*}
		or
		\begin{equation*}
			y^*_{i,j}=\begin{cases}
				g^{\mathcal{A}}_{i,j}&\mbox{if $g^{\mathcal{A}}_{i,j}< g^{\mathcal{C}}_{i,j}$},\\
				+\infty&\mbox{otherwise},
			\end{cases}
			\quad\mbox{and}\quad
			z^*_{i,j}=\begin{cases}
				g^{\mathcal{C}}_{i,j}&\mbox{if $g^{\mathcal{C}}_{i,j}\leq g^{\mathcal{A}}_{i,j}$},\\
				+\infty&\mbox{otherwise},
			\end{cases}
		\end{equation*}
		depending on whichever between decrease in retention rate and increase in dividend rate is prioritized.\\
		\textbf{Case 5:} If $\max_{x\in[0,\infty)}G^{\mathcal{E}}_{i,j}(x)=\max_{x\in[0,\infty)}G^{\mathcal{A}}_{i,j}(x)>\max_{x\in[0,\infty)}G^{\mathcal{C}}_{i,j}(x)$, then
		\begin{equation*}
			y^*_{i,j}=\begin{cases}
				g^{\mathcal{E}}_{i,j}&\mbox{if $g^{\mathcal{E}}_{i,j}\leq g^{\mathcal{A}}_{i,j}$},\\
				g^{\mathcal{A}}_{i,j}&\mbox{otherwise},
			\end{cases}
			\quad\mbox{and}\quad
			z^*_{i,j}=\begin{cases}
				g^{\mathcal{E}}_{i,j}&\mbox{if $g^{\mathcal{E}}_{i,j}\leq g^{\mathcal{A}}_{i,j}$},\\
				+\infty&\mbox{otherwise}.
			\end{cases}
		\end{equation*}
		\textbf{Case 6:} If $\max_{x\in[0,\infty)}G^{\mathcal{E}}_{i,j}(x)=\max_{x\in[0,\infty)}G^{\mathcal{C}}_{i,j}(x)>\max_{x\in[0,\infty)}G^{\mathcal{A}}_{i,j}(x)$, then
		\begin{equation*}
			y^*_{i,j}=\begin{cases}
				g^{\mathcal{E}}_{i,j}&\mbox{if $g^{\mathcal{E}}_{i,j}\leq g^{\mathcal{C}}_{i,j}$},\\
				+\infty&\mbox{otherwise},
			\end{cases}
			\quad\mbox{and}\quad
			z^*_{i,j}=\begin{cases}
				g^{\mathcal{E}}_{i,j}&\mbox{if $g^{\mathcal{E}}_{i,j}\leq g^{\mathcal{C}}_{i,j}$},\\
				g^{\mathcal{C}}_{i,j}&\mbox{otherwise}.
			\end{cases}
		\end{equation*}
		\textbf{Case 7:} If $\max_{x\in[0,\infty)}G^{\mathcal{E}}_{i,j}(x)=\max_{x\in[0,\infty)}G^{\mathcal{A}}_{i,j}(x)=\max_{x\in[0,\infty)}G^{\mathcal{C}}_{i,j}(x)$, then
		\begin{equation*}
			y^*_{i,j}=\begin{cases}
				g^{\mathcal{E}}_{i,j}&\mbox{if $g^{\mathcal{E}}_{i,j}\leq g^{\mathcal{A}}_{i,j}\wedge g^{\mathcal{C}}_{i,j}$},\\
				g^{\mathcal{A}}_{i,j}&\mbox{if $g^{\mathcal{A}}_{i,j}< g^{\mathcal{E}}_{i,j}\wedge g^{\mathcal{C}}_{i,j}$}\\
                &\mbox{or $g^{\mathcal{A}}_{i,j}= g^{\mathcal{C}}_{i,j}< g^{\mathcal{E}}_{i,j}$},\\
				+\infty&\mbox{otherwise,}
			\end{cases}\quad\mbox{and}\quad
			z^*_{i,j}=\begin{cases}
				g^{\mathcal{E}}_{i,j}&\mbox{if $g^{\mathcal{E}}_{i,j}\leq g^{\mathcal{A}}_{i,j}\wedge g^{\mathcal{C}}_{i,j}$},\\
				g^{\mathcal{C}}_{i,j}&\mbox{if $g^{\mathcal{C}}_{i,j}< g^{\mathcal{E}}_{i,j}\wedge g^{\mathcal{A}}_{i,j}$},\\
				+\infty&\mbox{otherwise,}
			\end{cases}
		\end{equation*}
		or
		\begin{equation*}
			y^*_{i,j}=\begin{cases}
				g^{\mathcal{E}}_{i,j}&\mbox{if $g^{\mathcal{E}}_{i,j}\leq g^{\mathcal{A}}_{i,j}\wedge g^{\mathcal{C}}_{i,j}$},\\
				g^{\mathcal{A}}_{i,j}&\mbox{if $g^{\mathcal{A}}_{i,j}< g^{\mathcal{E}}_{i,j}\wedge g^{\mathcal{C}}_{i,j}$},\\
				+\infty&\mbox{otherwise,}
			\end{cases}\quad\mbox{and}\quad
			z^*_{i,j}=\begin{cases}
				g^{\mathcal{E}}_{i,j}&\mbox{if $g^{\mathcal{E}}_{i,j}\leq g^{\mathcal{A}}_{i,j}\wedge g^{\mathcal{C}}_{i,j}$},\\
				g^{\mathcal{C}}_{i,j}&\mbox{if $g^{\mathcal{C}}_{i,j}< g^{\mathcal{E}}_{i,j}\wedge g^{\mathcal{A}}_{i,j}$}\\
                &\mbox{or $g^{\mathcal{C}}_{i,j}= g^{\mathcal{A}}_{i,j}< g^{\mathcal{E}}_{i,j}$},\\
				+\infty&\mbox{otherwise,}
			\end{cases}
		\end{equation*}
		depending on whichever between decrease in retention rate and increase in dividend rate is prioritized.

	To establish continuity and differentiability of $\wopt$, we consider three auxiliary problems of looking for the smallest solution $U^*_{\mathcal{A}}$, $U^*_{\mathcal{C}}$, and $U^*_{\mathcal{E}}$ of the equation $\mathcal{L}^{a_i,c_j}(U)=0$ in $[0,\infty)$ with boundary condition $U(0)=0$ above $\wopt(\cdot,a_{i+1},c_j)$, $\wopt(\cdot,a_{i},c_{j+1})$, and $\wopt(\cdot,a_{i+1},c_{j+1})$, respectively. Once we find ${U}^*_{\mathcal{A}}$, ${U}^*_{\mathcal{C}}$, and ${U}^*_{\mathcal{E}}$, we get, respectively,
	\begin{equation*}
		y^*_{i,j}=\begin{cases}
			0&\mbox{if $U^*_{\mathcal{A}}(\cdot)>\wopt(\cdot,a_{i+1},c_j)$ in $(0,\infty)$},\\
			\inf\{y>0:U^*_{\mathcal{A}}(\cdot)=\wopt(\cdot,a_{i+1},c_j)\}&\mbox{otherwise},
		\end{cases}
	\end{equation*}
	\begin{equation*}
		z^*_{i,j}=\begin{cases}
			0&\mbox{if $U^*_{\mathcal{C}}(\cdot)>\wopt(\cdot,a_{i},c_{j+1})$ in $(0,\infty)$},\\
			\inf\{y>0:U^*_{\mathcal{C}}(\cdot)=\wopt(\cdot,a_{i},c_{j+1})\}&\mbox{otherwise},
		\end{cases}
	\end{equation*}\normalsize
	and
	\begin{equation*}
		y^*_{i,j}=z^*_{i,j}=\begin{cases}
			0\qquad\qquad\qquad\qquad\qquad\qquad\qquad\,\qquad\mbox{if $U^*_{\mathcal{E}}(\cdot)>\wopt(\cdot,a_{i+1},c_{j+1})$ in $(0,\infty)$},\\
			\inf\{y>0:U^*_{\mathcal{E}}(\cdot)=\wopt(\cdot,a_{i+1},c_{j+1})\}\quad\mbox{otherwise}.
		\end{cases}
	\end{equation*}\normalsize
	We then have 
	\begin{equation*}
		\begin{cases}
			\wopt(x,a_i,c_j)=U^*_{\mathcal{A}}(x)&\mbox{for $x<y^*_{i,j}$},\\
			\wopt(x,a_i,c_j)=\wopt(x,a_{i+1},c_j)&\mbox{for $x\geq y^*_{i,j}$},
		\end{cases}
	\end{equation*}
	\begin{equation*}
		\begin{cases}
			\wopt(x,a_i,c_j)=U^*_{\mathcal{C}}(x)&\mbox{for $x<z^*_{i,j}$},\\
			\wopt(x,a_i,c_j)=\wopt(x,a_{i},c_{j+1})&\mbox{for $x\geq z^*_{i,j}$},
		\end{cases}
	\end{equation*}
	and
	\begin{equation*}
		\begin{cases}
			\wopt(x,a_i,c_j)=U^*_{\mathcal{E}}(x)&\mbox{for $x<y^*_{i,j}=z^*_{i,j}$},\\
			\wopt(x,a_i,c_j)=\wopt(x,a_{i+1},c_{j+1})&\mbox{for $x\geq y^*_{i,j}=z^*_{i,j}$}.
		\end{cases}
	\end{equation*}
	By Lemma \ref{lemma on existence of U}, ${U}^*_{\mathcal{A}}$, ${U}^*_{\mathcal{C}}$, and ${U}^*_{\mathcal{E}}$ exist.

    \begin{myrem}
         The construction of the optimal threshold functions ensures that $y^{*}_{i,j}\wedge z^{*}_{i,j}<\infty$; that is, the optimal threshold strategy always exists. Moreover, the strategy depends on which of the functions $G^{\mathcal{A}}_{i,j}$, $G^{\mathcal{C}}_{i,j}$, and $G^{\mathcal{E}}_{i,j}$ maximizes the expected payoff $W^{y,z}$ at the current levels of reserve, retention, and dividend. For instance, if the maximum value of $G^{\mathcal{A}}_{i,j}$ is the higher than the other two functions, the next threshold to consider would be the reinsurance level. In the case of a tie among these functions, the insurer must decide whether to increase the retention level or the dividend level, based on their strategic priorities. 
    \end{myrem}

    \begin{myrem}\label{remark on extended threshold strategy}
	Given $y_{i,j},z_{i,j}:\mathscr{A}\times\mathscr{C}\to[0,\infty]$, we have defined in \eqref{threshold strategy} a threshold strategy $\pi^{y,z}:=\{(A_{x,a_i,c_j},C_{x,a_i,c_j})\}_{(x,a_i,c_j)\in[0,\infty)\times\mathscr{A}\times\mathscr{C}}$, where $(A_{x,a_i,c_j},C_{x,a_i,c_j})\in\Pi^{\mathscr{A},\mathscr{C}}_{x,a_i,c_j}$ for $1\leq i\leq m$ and $1\leq j\leq n$. This strategy can be extended to
	\begin{align*}
		\tilde{\pi}^{y,z}:=\{(A_{x,a,c},C_{x,a,c})\}_{(x,a,c)\in[0,\infty)\times[a_m,a_1]\times[c_1,c_n]}
	\end{align*}
	as follows:
	\begin{itemize}
		\item[(i)] If $a\in(a_{i+1},a_{i})$, $c\in(c_j,c_{j+1})$, and $x<y_{i,j}\wedge z_{i,j}$, retain proportion $a$ of the incoming claims and pay dividends at rate $c$ until the time of ruin. If the current reserve reaches $y_{i,j}$ before ruin time, follow $\left(A_{y_{i,j},a_{i+1},c},C_{y_{i,j},a_{i+1},c}\right)\in\Pi^{\mc{A},\mc{C}}_{x,a_{i+1},c}$. If the current reserve reaches $z_{i,j}$ first before ruin time, follow $\left(A_{z_{i,j},a,c_{j+1}},C_{z_{i,j},a,c_{j+1}}\right)\in\Pi^{\mc{A},\mc{C}}_{x,a,c_{j+1}}$. If the current reserve reaches $y_{i,j}$ and $z_{i,j}$ simultaneously before ruin time, follow $\left(A_{y_{i,j},a_{i+1},c_{j+1}},C_{y_{i,j},a_{i+1},c_{j+1}}\right)\in\Pi^{\mc{A},\mc{C}}_{x,a_{i+1},c_{j+1}}$.
		
		\item[(ii)] If $a\in(a_{i+1},a_{i})$, $c\in(c_j,c_{j+1})$, and $x\geq y_{i,j}$, follow $\left(A_{x,a_{i+1},c_j},C_{x,a_{i+1},c_j}\right)\in\Pi^{\mc{A},\mc{C}}_{x,a_{i+1},c_j}$. More precisely, if $(x,a,c)\in[0,y_{i,j})\times(a_{i+1},a_i)\times(c_j,c_{j+1})$ with $y_{i,j}<z_{i,j}$, then $(A_{x,a},C_{x,c})\in\Pi^{\mc{A},\mc{C}}_{x,a,c}$ is defined as $(A_{x,a,c},C_{x,a,c})(t)=(a,c)$ and so $X^{(A_{x,a,c},C_{x,a,c})}_t=X_t-ct$ for 
		\begin{align*}
			t<\tau_{\pi}\wedge\tau^{\mc{A}}_i\quad\mbox{where $\tau^{\mc{A}}_i:=\min\{s:X^{(A_{x,a,c},C_{x,a,c})}_s=y_{i,j}\}$},
		\end{align*}
		and $(A_{x,a,c},C_{x,a,c})(t)=(A_{y_{i,j},a_{i+1},c_j},C_{y_{i,j},a_{i+1},c_j})({t-\tau^{\mc{A}}_i})\in\Pi^{\mc{A},\mc{C}}_{y_{i,j},a_{i+1},c_j}$ for $t\geq\tau^{\mc{A}}_i$. Finally, it holds that $(A_{x,a,c},C_{x,a,c})=(A_{x,a_{i+1},c_j},C_{x,a_{i+1},c_j})\in\Pi^{\mc{A},\mc{C}}_{x,a_{i+1},c_j}$ for $(x,a,c)\in[y_{i,j},\infty)\times(a_{i+1},a_i)\times(c_j,c_{j+1})$.
		
		\item[(iii)] If $a\in(a_{i+1},a_{i})$, $c\in(c_{j},c_{j+1})$ and $x\geq z_{i,j}$, follow $\left(A_{x,a_i,c_{j+1}},C_{x,a_i,c_{j+1}}\right)\in\Pi^{\mc{A},\mc{C}}_{x,a_i,c_{j+1}}$. More precisely, if $(x,a,c)\in[0,z_{i,j})\times(a_{i+1},a_i)\times(c_j,c_{j+1})$ with $z_{i,j}<y_{i,j}$, then $(A_{x,a,c},C_{x,a,c})\in\Pi^{\mc{A},\mc{C}}_{x,a,c}$ is defined as $(A_{x,a,c},C_{x,a,c})(t)=(a,c)$ and so $X^{(A_{x,a,c},C_{x,a,c})}_t=X_t-ct$ for 
		\begin{align*}
			t<\tau_{\pi}\wedge\tau^{\mc{C}}_j\quad\mbox{where $\tau^{\mc{C}}_j:=\min\{s:X^{(A_{x,a,c},C_{x,a,c})}_s=z_{i,j}\}$},
		\end{align*}
		and $(A_{x,a,c},C_{x,a,c})(t)=(A_{z_{i,j},a_{i},c_{j+1}},C_{z_{i,j},a_i,c_{j+1}})(t-\tau^{\mc{C}}_j)\in\Pi^{\mc{A},\mc{C}}_{y_{i,j},a_{i},c_{j+1}}$ for $t\geq\tau^{\mc{C}}_j$. Finally, it holds that $(A_{x,a,c},C_{x,a,c})=(A_{x,a_{i},c_{j+1}},C_{x,a_i,c_{j+1}})\in\Pi^{\mc{A},\mc{C}}_{x,a_{i},c_{j+1}}$ for $(x,a,c)\in[z_{i,j},\infty)\times(a_{i+1},a_i)\times(c_j,c_{j+1})$.
		
		\item[(iv)] If $a\in(a_{i+1},a_{i})$, $c\in(c_{j},c_{j+1})$ and $x\geq y_{i,j}=z_{i,j}$, follow $\left(A_{x,a_{i+1},c_{j+1}},C_{x,a_{i+1},c_{j+1}}\right)\in\Pi^{\mc{A},\mc{C}}_{x,a_{i+1},c_{j+1}}$. More precisely, if $(x,a,c)\in[0,z_{i,j})\times(a_{i+1},a_i)\times(c_j,c_{j+1})$ with $z_{i,j}=y_{i,j}$, then $(A_{x,a,c},C_{x,a,c})\in\Pi^{\mc{A},\mc{C}}_{x,a,c}$ is defined as $(A_{x,a,c},C_{x,a,c})(t)=(a,c)$ and so $X^{(A_{x,a,c},C_{x,a,c})}(t)=X_t-ct$ for 
		\begin{align*}
			t<\tau_{\pi}\wedge\tau^{\mc{E}}_{i,j}\quad\mbox{where $\tau^{\mc{E}}_{i,j}:=\min\{s:X^{(A_{x,a,c},C_{x,a,c})}_s=z_{i,j}\}$},
		\end{align*}
		and $(A_{x,a,c},C_{x,a,c})(t)=(A_{z_{i,j},a_{i+1},c_{j+1}},C_{z_{i,j},a_{i+1},c_{j+1}})(t-\tau^{\mc{E}}_{i,j})\in\Pi^{\mc{A},\mc{C}}_{y_{i,j},a_{i+1},c_{j+1}}$ for $t\geq\tau^{\mc{E}}_{i,j}$. Finally, $(A_{x,a,c},C_{x,a,c})=(A_{x,a_{i+1},c_{j+1}},C_{x,a_{i+1},c_{j+1}})\in\Pi^{\mc{A},\mc{C}}_{x,a_{i+1},c_{j+1}}$ for $(x,a,c)\in[z_{i,j},\infty)\times(a_{i+1},a_i)\times(c_j,c_{j+1})$.
	\end{itemize}
	The value function of the extended stationary strategy $\tilde\pi^{y,z}$ is defined as
	\begin{align*}
		J^{\tilde{\pi}^{y,z}}(x,a,c):=J(x;A_{x,a,c},C_{x,a,c}):[0,\infty)\times[a_m,a_1]\times[c_1,c_n]\to\mathbb{R}.
	\end{align*}
\end{myrem}

\section{Numerical Illustrations}\label{section:examples}

In this section, we present several numerical examples to gain further insight into the optimal dividend and reinsurance strategies, as well as the optimal value function. These examples are intended to provide a clearer understanding of the optimal retention and dividend threshold levels and how they affect the optimal value function.

Figure \ref{figure:global} presents three examples, each illustrating the optimal value function. The vertical dotted lines represent the optimal threshold levels. We fix the following parameters across all three examples: $\sigma=1.5$, $b=2$, and $q=0.1$.

Consider the case $\mu=6$, $\mathscr{A}=\{0.80,0.90\}$, and $\mathscr{C}=\{2,4\}$. We obtain the following optimal threshold levels: $z^*_{1,1}=13.04$ and $y^*_{1,2}=348.5$. The sequence of retention-dividend rate changes is as follows: $(0.9,2)\stackrel{x=13.04}{\longrightarrow}(0.9,4)\stackrel{x=348.5}{\longrightarrow}(0.8,4)$, and this is illustrated in Figures \ref{fig:plot-1-v1} and \ref{fig:plot-1-v2}. The optimal strategy is to increase the dividend level first and then decrease the retained proportion of incoming claims if ruin time does not occur before or within these level changes. In Figure \ref{fig:plot-1-v2}, we can see that $W^{y^*,z^*}(x,a_1,c_1)$ (in red) is greater than $W^{y^*,z^*}(x,a_1,c_2)$ (in blue) when $x<z^*_{1,1}$. We can also see here how concavity is not guaranteed when both the minimum retention level and the maximum dividend rate have not been attained.  

Next, we consider a similar setup with a higher adjusted mean $\mu=10$. We obtain the following optimal threshold levels: $y^*_{1,1}=0$ and $z^*_{2,1}=1.92$. The sequence of retention-dividend level changes is as follows:  $(0.90,2)\stackrel{x=0}{\longrightarrow}(0.80,2)\stackrel{x=1.92}{\longrightarrow}(0.80,4)$, and this is illustrated in Figure \ref{fig:plot-2-v1}. The curves $W^{y^*,z^*}(x,a_1,c_1)$ and $W^{y^*,z^*}(x,a_2,c_1)$ coincide since it is optimal to increase the reinsurance level right away assuming that the initial reserve level is positive. The dividend level is then increased as soon as the reserve reaches the threshold level $z^*_{2,1}=1.92$. It is also interesting to note that in the second setup, we have $\mu a_1-2(b+c_1)<0$ and $\sigma<\sqrt{\frac{-\mu(\mu a_1-2(b+c_1))}{2qa_1}}$, which implies that $\frac{\partial}{\partial a}\theta_2(a,c)<0$. In the first setup, we have $\mu a_1-2(b+c_1)>0$, which implies that $\frac{\partial}{\partial a}\theta_2(a,c)>0$.

\begin{figure}[hbt!]
	\begin{subfigure}{.5\textwidth}
		\centering
		\includegraphics[width=1\textwidth]{"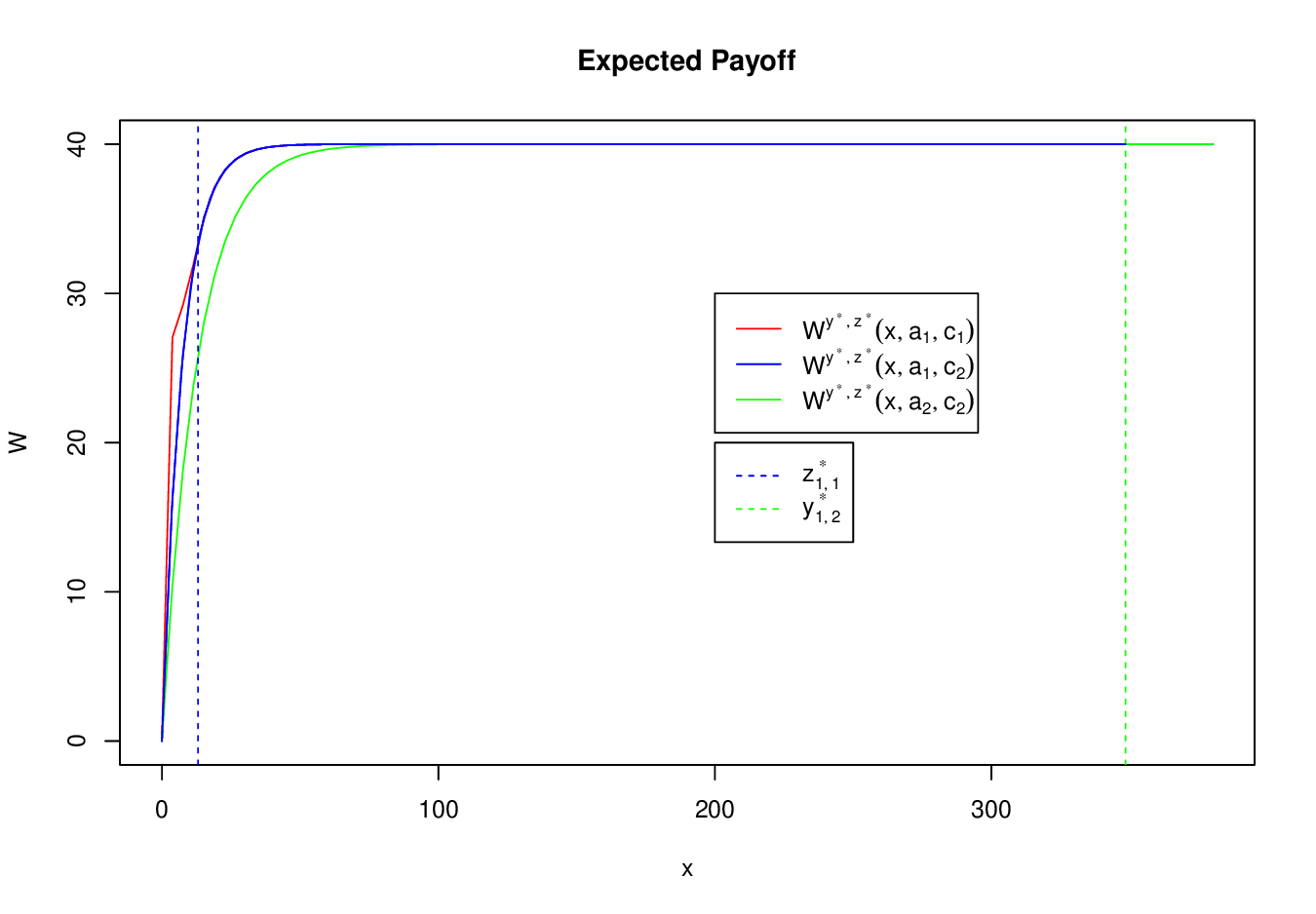"}
		\caption{Example 1: a zoomed-out perspective}
		\label{fig:plot-1-v1}
	\end{subfigure}%
	\begin{subfigure}{.5\textwidth}
		\centering
		\includegraphics[width=1\textwidth]{"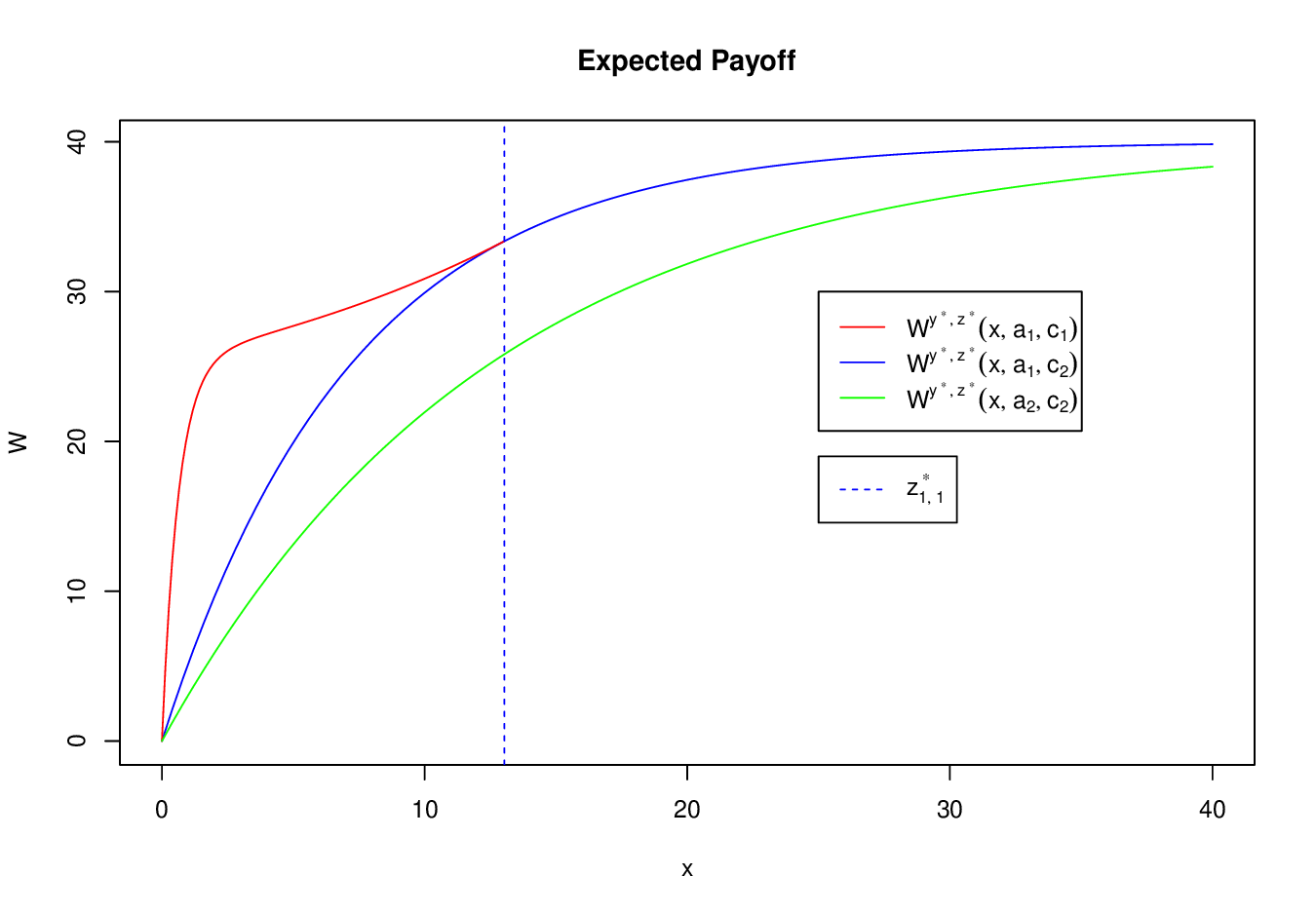"}
		\caption{Example 1: a zoomed-in perspective}
		\label{fig:plot-1-v2}
	\end{subfigure}
	\begin{subfigure}{.5\textwidth}
		\centering
		\includegraphics[width=1\textwidth]{"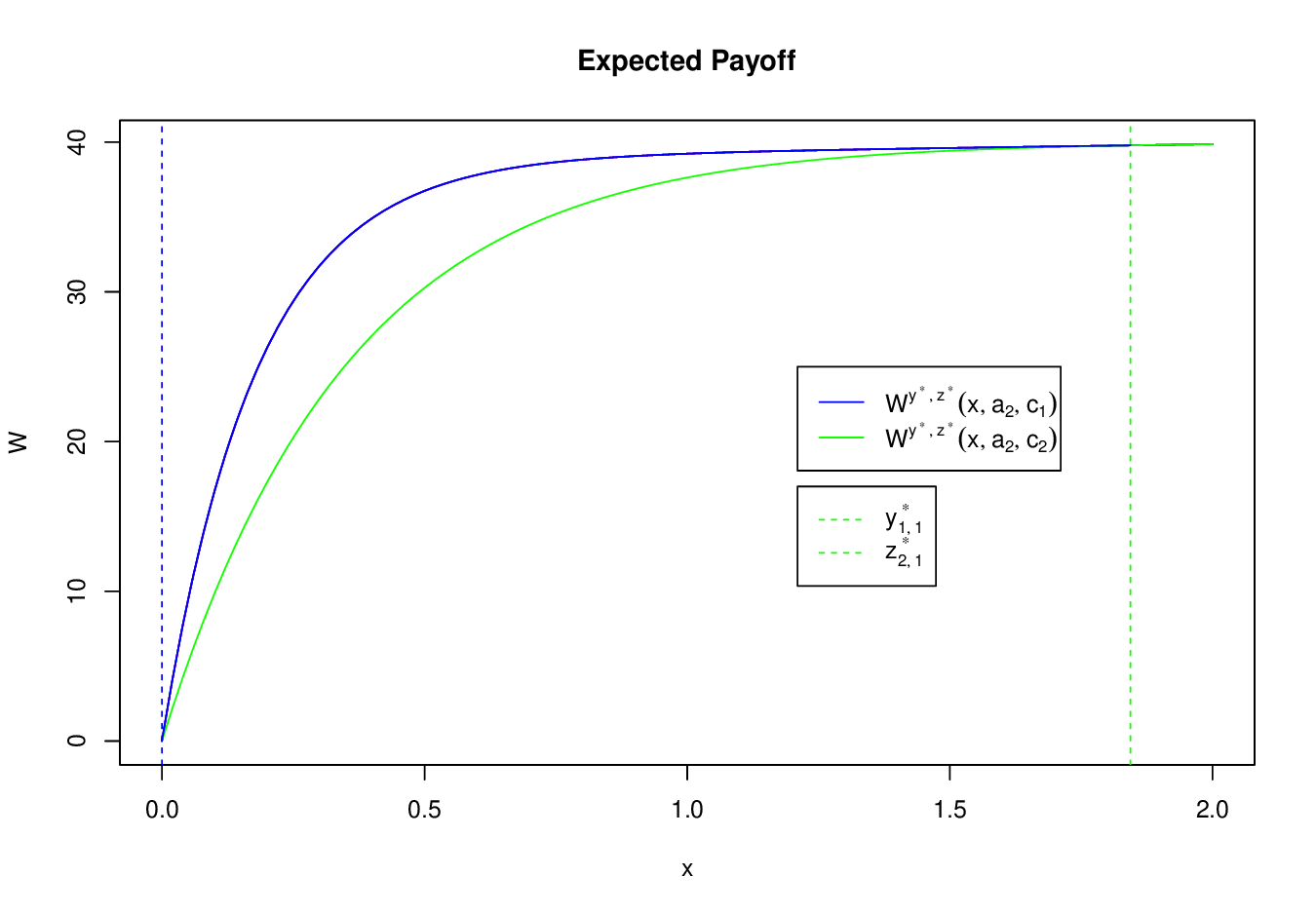"}
		\caption{Example 2}
		\label{fig:plot-2-v1}
	\end{subfigure}%
	\begin{subfigure}{.5\textwidth}
		\centering
		\includegraphics[width=1\textwidth]{"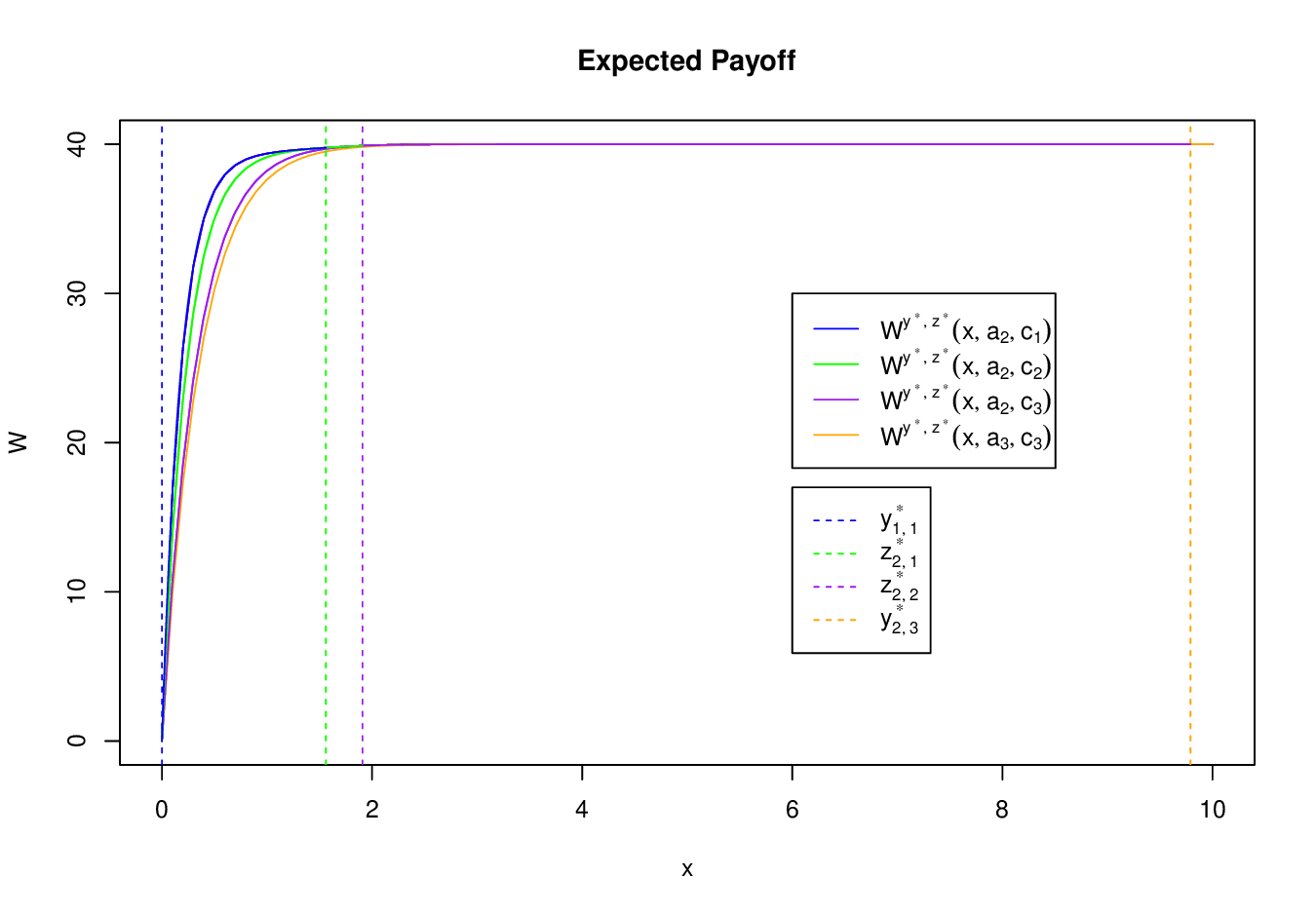"}
		\caption{Example 3: a zoomed-out perspective}
		\label{fig:plot-3-v1}
	\end{subfigure}
	\begin{subfigure}{.5\textwidth}
		\centering
		\includegraphics[width=1\textwidth]{"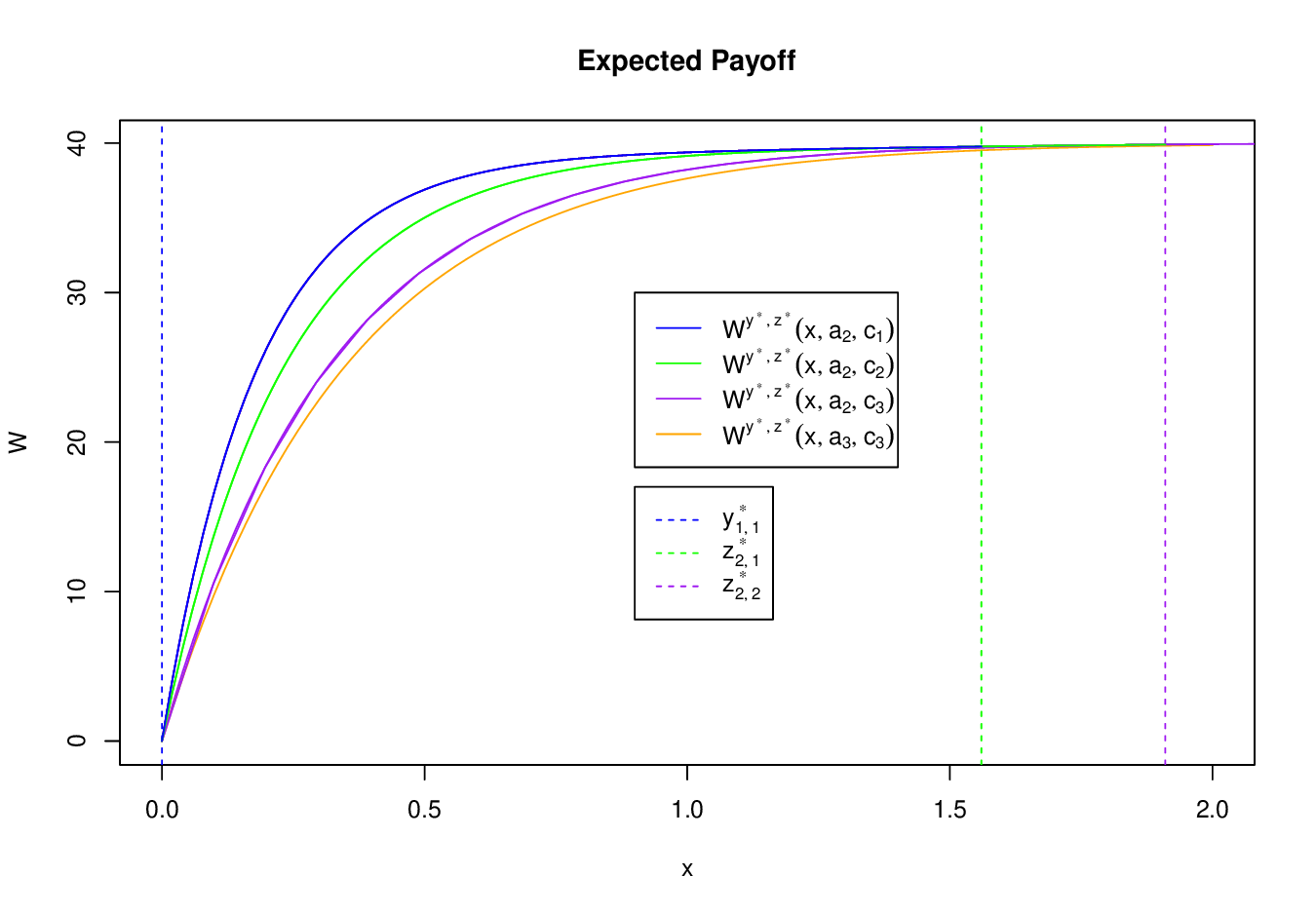"}
		\caption{Example 3: a zoomed-in perspective}
		\label{fig:plot-3-v2}
	\end{subfigure}
 \caption{Examples 1, 2, and 3}
 \label{figure:global}
\end{figure}

Finally, consider the case where $\mu=10$, $\mathscr{A}=\{0.80,0.85,0.90\}$, and $\mathscr{C}=\{2,3,4\}$. As in the previous examples, the maximum dividend rate and the minimum retention level remain unchanged. The main difference is the presence of an intermediate level between the extremes. We obtain the following optimal threshold levels: $y^*_{1,1}=0$, $z^*_{2,1}=1.56$, $z^*_{2,2}=1.91$, and $y^*_{2,3}=9.79$. The sequence of retention-dividend level changes is as follows: $(0.90,2)\stackrel{x=0}{\longrightarrow}(0.85,2)\stackrel{x=1.56}{\longrightarrow}(0.85,3)\stackrel{x=1.91}{\longrightarrow}(0.85,4)\stackrel{x=9.79}{\longrightarrow}(0.80,4)$, and this is illustrated in Figures \ref{fig:plot-3-v1} and \ref{fig:plot-3-v2}. As in the second example, the optimal strategy is to increase the reinsurance level immediately. The dividend level is then increased to its maximum before the reinsurance level is maximized, provided that ruin does not occur during these level changes.

\begin{table}[hbt!]
    \centering
    \begin{tabular}{cc}
        \begin{subtable}[t]{0.45\linewidth}
            \centering
            
            \begin{tabular}{@{}cccccc@{}}
\toprule
$\mu$ & $y_{1,1}^*$ & $z^*_{1,1}$ & $y_{1,2}^*$ & $z_{2,1}^*$ &  Strategy \\ \midrule
2   &$\infty$&0&1209.44&N/A&D-R\\
3   &$\infty$&10.67&1000.02&N/A&D-R\\
4   & $\infty$  & 16.98  & 782.9  & N/A  & D-R\\
6   &$\infty$& 13.04& 348.49& N/A& D-R\\
8   &$\infty$&3.24&37.33&N/A&D-R\\
10  &0&$\infty$&N/A&1.92&R-D\\
\bottomrule
\end{tabular}
\caption{Sensitivity analysis for $\mu$}
\label{table:mu}
        \end{subtable} \vspace{1em} &
        \begin{subtable}[t]{0.45\linewidth}
            \centering
            \begin{tabular}{@{}cccccc@{}}
\toprule
$\sigma$ & $y_{1,1}^*$ & $z^*_{1,1}$ & $y_{1,2}^*$ & $z_{2,1}^*$ &  Strategy \\ \midrule
0.375   &$\infty$&6.15&333.66&N/A&D-R\\
 0.75  & $\infty$  & 11.86  & 336.73  & N/A  & D-R\\
1.5   &$\infty$& 13.04& 348.49& N/A& D-R\\
3   &$\infty$&16.24&389.44&N/A&D-R\\
6  &$\infty$&20.4&507.03&N/A&D-R\\
16  &0&$\infty$&N/A&6.72&R-D\\
\bottomrule
\end{tabular}
            \caption{Sensitivity analysis for $\sigma$}
            \label{table:sigma}
        \end{subtable} \\
        \vspace{1em}
        
        \begin{subtable}[t]{0.45\linewidth}
            \centering
            \begin{tabular}{@{}cccccc@{}}
\toprule
$b$ & $y_{1,1}^*$ & $z^*_{1,1}$ & $y_{1,2}^*$ & $z_{2,1}^*$ &  Strategy \\ \midrule
0   &0&$\infty$&N/A&3.48&R-D\\
1   &$\infty$&5.62&107.1&N/A&D-R\\
2   & $\infty$  & 13.04  & 348.49  & N/A  & D-R\\
4   &$\infty$& 16.48& 893.2& N/A& D-R\\
8  &$\infty$&0&1992.24&N/A&D-R\\
\bottomrule
\end{tabular}
            \caption{Sensitivity analysis for $b$}
            \label{table:b}
        \end{subtable} \vspace{1em} &
        \begin{subtable}[t]{0.45\linewidth}
            \centering
            \begin{tabular}{@{}cccccc@{}}
\toprule
$q$ & $y_{1,1}^*$ & $z^*_{1,1}$ & $y_{1,2}^*$ & $z_{2,1}^*$ &  Strategy \\ \midrule
0.025   &$\infty$&47.43&1414.25&N/A&D-R\\
 0.05  & $\infty$  & 24.56  & 698.51  & N/A  & D-R\\
0.1   &$\infty$& 13.04& 348.49& N/A& D-R\\
0.5   &$\infty$&3.40&75.64&N/A&D-R\\
1  &$\infty$&1.93&41.66&N/A&D-R\\
\bottomrule
\end{tabular}
\caption{Sensitivity analysis for $q$}
\label{table:q}
        \end{subtable} \\\vspace{1em}

        \begin{subtable}[t]{0.45\linewidth}
            \centering
            \begin{tabular}{@{}cccccc@{}}
\toprule
$a_1$ & $y_{1,1}^*$ & $z^*_{1,1}$ & $y_{1,2}^*$ & $z_{2,1}^*$ &  Strategy \\ \midrule
0.81   &$\infty$&15.26&344.36&N/A&D-R\\
 0.85  & $\infty$  & 15.26  & 348.49  & N/A  & D-R\\
0.9   &$\infty$& 13.04& 348.49& N/A& D-R\\
0.95   &$\infty$&11.03&348.49&N/A&D-R\\
1  &$\infty$&8.91&348.49&N/A&D-R\\
\bottomrule
\end{tabular}
\caption{Sensitivity analysis for $a_1$}
\label{table:a_1}
        \end{subtable} \vspace{1em} &
        \begin{subtable}[t]{0.45\linewidth}
            \centering
            
\begin{tabular}{@{}cccccc@{}}
\toprule
$a_2$ & $y_{1,1}^*$ & $z^*_{1,1}$ & $y_{1,2}^*$ & $z_{2,1}^*$ &  Strategy \\ \midrule
0.1   &$\infty$&13.04&1496.84&N/A&D-R\\
0.5   &$\infty$&13.04&834.14&N/A&D-R\\
0.8   & $\infty$  & 13.04  & 348.49  & N/A  & D-R\\
0.85   &$\infty$& 13.04& 272.39& N/A& D-R\\
0.89  &$\infty$&13.04&213.53&N/A&D-R\\
\bottomrule
\end{tabular}
            \caption{Sensitivity analysis for $a_2$}
            \label{table:a_2}
        \end{subtable} \\\vspace{1em}

        \begin{subtable}[t]{0.45\linewidth}
            \centering
            \begin{tabular}{@{}cccccc@{}}
\toprule
$c_1$ & $y_{1,1}^*$ & $z^*_{1,1}$ & $y_{1,2}^*$ & $z_{2,1}^*$ &  Strategy \\ \midrule
0   &$\infty$&12.66&348.49&N/A&D-R\\
 1  & $\infty$  & 12.77  & 348.49  & N/A  & D-R\\
2   &$\infty$& 13.04& 348.49& N/A& D-R\\
3   &$\infty$&14.21&348.49&N/A&D-R\\
3.5  &$\infty$&16.47&348.49&N/A&D-R\\
\bottomrule
\end{tabular}
\caption{Sensitivity analysis for $c_1$}
\label{table:c_1}
        \end{subtable} \vspace{1em} &
        \begin{subtable}[t]{0.45\linewidth}
            \centering
            \begin{tabular}{@{}cccccc@{}}
\toprule
$c_2$ & $y_{1,1}^*$ & $z^*_{1,1}$ & $y_{1,2}^*$ & $z_{2,1}^*$ &  Strategy \\ \midrule
2.5   &$\infty$&3.89&42.9&N/A&D-R\\
3   &$\infty$&5.60&105.98&N/A&D-R\\
4   & $\infty$  & 13.04  & 348.49  & N/A  & D-R\\
6   &$\infty$& 22.31& 906.18& N/A& D-R\\
8  &$\infty$&25.99&1479.92&N/A&D-R\\
\bottomrule
\end{tabular}
            \caption{Sensitivity analysis for $c_2$}
            \label{table:c_2}
        \end{subtable} \\\vspace{1em}

        \begin{subtable}[t]{0.5\linewidth}
            \centering
            \begin{tabular}{@{}cccccc@{}}
\toprule
$(a_1,a_2)$ & $y_{1,1}^*$ & $z^*_{1,1}$ & $y_{1,2}^*$ & $z_{2,1}^*$ &  Strategy \\ \midrule
$(0.5,0.4)$   &$\infty$&14.28&999.09&N/A&D-R\\
$(0.7,0.6)$   &$\infty$&15.93&669.88&N/A&D-R\\
$(0.9,0.8)$  & $\infty$  & 13.04  & 348.49  & N/A  & D-R\\
$(0.95,0.85)$  &$\infty$& 11.03& 272.39& N/A& D-R\\
$(1,0.9)$  &$\infty$&8.91&201.12&N/A&D-R\\
\bottomrule
\end{tabular}
            \caption{Sensitivity analysis for $(a_1,a_2)$}
            \label{table:a_1-a_2}
        \end{subtable} \vspace{1em} &
        \begin{subtable}[t]{0.5\linewidth}
            \centering
            \begin{tabular}{@{}cccccc@{}}
\toprule
$(c_1,c_2)$ & $y_{1,1}^*$ & $z^*_{1,1}$ & $y_{1,2}^*$ & $z_{2,1}^*$ &  Strategy \\ \midrule
$(0,2)$   &0&$\infty$&N/A&2.92&R-D\\
 $(1,3)$  & $\infty$  & 5.16  & 105.98  & N/A  & D-R\\
$(2,4)$   &$\infty$& 13.04& 348.49& N/A& D-R\\
$(3,5)$   &$\infty$&20&623.72&N/A&D-R\\
$(4,6)$  &$\infty$&29.07&906.18&N/A&D-R\\
\bottomrule
\end{tabular}
\caption{Sensitivity analysis for $(c_1,c_2)$}
\label{table:c_1-c_2}
        \end{subtable}
        \end{tabular}
    \caption{Comparison of optimal threshold levels with base case: $\mu=6$, $\sigma=1.5$, $b=2$, $q=0.1$, $a_1=0.9$, $a_2=0.8$, $c_1=2$, and $c_2=4$}
\end{table}

{Following the numerical examples, we now investigate the impact of the model parameters on the optimal retention and dividend threshold strategies through a sensitivity analysis. For our analysis, we use the model parameters from Example 1 of Figure \ref{fig:plot-1-v1} as our base case: $\mu=6$, $\sigma=1.5$, $b=2$, $q=0.1$, $a_1=0.9$, $a_2=0.8$, $c_1=2$, and $c_2=4$. The sensitivity analysis varies the following parameters: the adjusted mean $\mu$, the adjusted volatility $\sigma$, the nonhomogeneous term $b$, the discount rate $q$, the retention levels $a_1$ and $a_2$, and the dividend rates $c_1$ and $c_2$.

In each table, the columns for $y_{1,1}^*$ and $z_{1,1}^*$ represent the optimal retention and threshold levels given the initial retention level $a_1$ and the initial dividend rate $c_1$. If $y_{1,1}^*<\infty$ and $z_{1,1}^*=\infty$, then the retention level must be decreased first. Otherwise, the dividend rate must be increased first. The column for $y_{1,2}^*$ represents the optimal retention threshold level following an increase in the dividend rate, while the column for $z_{2,1}^*$ represents the optimal dividend threshold level after a decrease in the retention level. A value of ``N/A" for $y_{1,2}^*$ means that the retention level has already been decreased, whereas an ``N/A" for $z_{2,1}^*$ means that the dividend rate has already been increased. The rightmost column represents the optimal strategy: ``D-R" signifies that the dividend rate should be increased first, followed by a decrease in the retention level, whereas ``R-D" signifies that the retention level must be decreased first, followed by an increase in the dividend rate.

It can be observed that there is an inverse relationship between the optimal reinsurance threshold level $y_{1,2}^*$ and the parameters $\mu$, $q$, $a_2$, and $(a_1,a_2)$ (refer to Tables \ref{table:mu}, \ref{table:q}, \ref{table:a_2}, and \ref{table:a_1-a_2}). A higher $\mu$ implies a higher drift for the reserve process, to which the downward trend of $y_{1,2}^*$ can be attributed. As $q$ rises, the incentive to hold more reserves decreases because the future ``costs" associated with these reserves are also less significant. When $a_2$ is higher, the change in retention level from $a_1=0.90$ becomes larger. The same holds when the pair $(a_1,a_2)$ has higher values.

In contrast, a direct relationship can be observed between $y_{1,2}^*$ and the parameters $\sigma$, $b$, $c_2$, and $(c_1,c_2)$ (refer to Tables \ref{table:sigma}, \ref{table:b}, \ref{table:c_2}, and \ref{table:c_1-c_2}). A higher $\sigma$ implies a higher diffusion for the reserve process, which corresponds to higher volatility. As either $b$ or $c_2$ increases, the drift for the reserve process decreases, which makes lowering the retention level less advantageous. This trend also holds if both $c_1$ and $c_2$ increase simultaneously.

For the optimal dividend threshold level $z_{1,1}^*$, there is an inverse relationship with the parameters $q$, $a_1$, and $(a_1,a_2)$ (see Tables \ref{table:q}, \ref{table:a_1}, and \ref{table:a_1-a_2}). As $q$ increases, the valuation of future dividend flows decreases, leading to an optimal strategy of increasing the dividend rate at lower threshold levels. As $a_1$ increases, the \emph{expected} reserve level $\mathbb E[X_t]$ also increases, making it more attractive to increase the dividend rate at a lower threshold level. The same can be said when both $a_1$ and $a_2$ increase simultaneously.

On the other hand, $z_{1,1}^*$ exhibits a direct relationship with the parameters $\sigma$, $c_1$, $c_2$, and $(c_1,c_2)$ (see Tables \ref{table:sigma}, \ref{table:c_1}, \ref{table:c_2}, and \ref{table:c_1-c_2}). Higher dividend rates make increasing the dividend rate less attractive. A smaller $\sigma$ implies lower volatility, allowing a lower threshold level for a dividend increase. 

There is no noticeable trend for $z_{1,1}^*$ in Tables \ref{table:mu} and \ref{table:b}. It remains constant in Table \ref{table:a_2} since $z_{1,1}^*$ does not depend on the next retention level $a_2$. This is also similar to the stable behavior of $y_{1,2}^*$ with respect to $c_1$ in Table \ref{table:c_1}, which can be attributed to the fact that $y_{1,2}^*$ depends on $c_2$, and not on $c_1$.

The optimal strategy is predominantly to increase the dividends first and then decrease the retention level (i.e., D-R). However, a shift to the R-D strategy occurs in ``extreme" cases, as illustrated in Tables \ref{table:mu}, \ref{table:sigma}, \ref{table:b}, and \ref{table:c_1-c_2}.  If the adjusted mean $\mu$ is sufficiently large or if $b$ is relatively small,  the drift of the reserve process becomes large, making it more advantageous to decrease the retention level (or, equivalently, to increase the reinsurance coverage). Similarly, if $\sigma$ is excessively high, the resulting volatility in reserve levels necessitates an immediate decrease in the retention level. The same behavior is observed if both $c_1$ and $c_2$ are sufficiently small. It is also noteworthy that when the optimal strategy is R–D, the retention level is decreased immediately.}

\section{Conclusion}\label{section:conc}
In this paper, we have extended the literature on optimal dividend payout problems with ratcheting by incorporating an irreversible reinsurance constraint, under which the reinsurance and dividend levels are restricted to a finite set. Through a dynamic programming approach, we have derived the HJB equation. We have presented a method to determine the optimal threshold levels for the dividend and reinsurance variables. The sensitivity analysis suggests that it is almost always optimal to increase the dividends before increasing the reinsurance coverage. A possible direction for future research is to consider jump-diffusion processes on top of the Brownian motion to model the reserve level. Drawdown constraints, which allow the levels to decrease by a fixed proportion of the current level, could also be incorporated for both dividend and reinsurance levels.




\appendix
\section{Proofs of Section \ref{section:model}, and Intermediate Lemmas and Propositions}\label{appendix:a}

\begin{proof}[Proof of Proposition \ref{bound and monotone}]
	From Remark \ref{remark on no constraint case}, since $V$ is bounded above by $V_{NC}$ and that $\lim_{x\to\infty}V_{NC}(x)=\frac{\ol c}{q}$, it follows that $V(x,a,c)$ is bounded by $\frac{\ol c}{q}$.

    The monotonicity of $V$ with respect to $a$ and $c$ follows from Remark \ref{remark on inclusion}. Given $0\leq x_1<x_2$, let $\pi_1:=(A^1,C^1)\in\Pi^{\mathscr{A},\mathscr{C}}_{x_1,a,c}$ be an admissible strategy for any $a\in\mathscr{A}$ and $c\in\mathscr{C}$. Moreover, define $\pi_2:=(A^2,C^2)\in\Pi^{\mathscr{A},\mathscr{C}}_{x_2,a,c}$ as
	\begin{equation*}
		{A^2(t):=
		\begin{cases}
			A^1(t),&\mbox{if $t<\tau_{\pi_1}$},\\
			\underline a,&\mbox{if $t\geq\tau_{\pi_1}$},
		\end{cases}}\quad\mbox{and}\quad
		C^2(t):=
		\begin{cases}
			C^1(t),&\mbox{if $t<\tau_{\pi_1}$},\\
			\overline c,&\mbox{if $t\geq\tau_{\pi_1}$},
		\end{cases}
	\end{equation*}
    where $\tau_{\pi_1}$ is the corresponding ruin time of the controlled process $\{X^{\pi_1}_t\}_{t\geq 0}$ with $X^{\pi_1}_0=x_1$.	Then,
	\begin{equation*}
    \begin{aligned}
		J(x_1;A^1,C^1)&=\mathbb{E}\left[\int_0^{\tau_{\pi_1}}e^{-qs}C^1(s)ds\right]\\
        &\leq\mathbb{E}\left[\int_0^{\tau_{\pi_1}}e^{-qs}C^{1}(s)ds\right]+\mathbb{E}\left[\int_{\tau_{\pi_1}}^{\tau_{\pi_2}}e^{-qs}\overline c\, ds\right]\\
        &=J(x_2;A^2,C^2),
        \end{aligned}
	\end{equation*}
	where $\tau_{\pi_2}$ is the corresponding ruin time of the controlled process $\{X^{\pi_2}_t\}_{t\geq 0}$ with $X^{\pi_2}_0=x_2$. It then implies that $V(x,a,c)$ is nondecreasing in $x$.
\end{proof}

\begin{proof}[Proof of Proposition \ref{Lipschitz property}]
	By Proposition \ref{bound and monotone}, $V(x,a,c)$ is nondecreasing in $x$ and $a$ and nonincreasing in $c$. Hence, we obtain the first inequality $0\leq V(x_2,a_1,c_1)-V(x_1,a_2,c_2)$ for all $0\leq x_1\leq x_2$, $a_1,a_2\in\mathscr{A}$ with $a_2\leq a_1$, and $c_1,c_2\in\mathscr{C}$ with $c_1\leq c_2$. For the second inequality, we divide the proof into three parts.
	
	(Part I) We want to show that there exists $L_1>0$ such that
	\begin{equation}\label{a2}
		V(x_2,a,c)-V(x_1,a,c)\leq L_1\left(x_2-x_1\right)
	\end{equation}
	for all $0\leq x_1\leq x_2$. Let $\epsilon_1>0$ and $\pi^1:=(A,C)\in\Pi^{\mathscr{A},\mathscr{C}}_{x_2,a,c}$ such that
	\begin{equation}\label{a3}
		J(x_2;A,C)\geq V(x_2,a,c)-\epsilon_1.
	\end{equation}
	Let $\tau_1$ be the ruin time of the associated process $\{X_t^{\pi^1}\}_{t\geq 0}$. Define $\tilde{\pi}:=(\tilde{A},\tilde{C})\in\Pi^{\mathscr{A},\mathscr{C}}_{x_1,a,c}$ as $\tilde{\pi}_t=\pi^1_t$. Let $\tilde{\tau}\leq\tau_1$ be the ruin time of the associated process $\{X^{\tilde{\pi}}_t\}_{t\geq 0}$. It then holds that $X_t^{\pi^1}-X_t^{\tilde{\pi}}=x_2-x_1$ for $t\leq\tilde{\tau}$. Using the properties of conditional expectations and a shift in stopping times, we obtain
	\begin{equation}\label{a3.5}
		\begin{aligned}
			J(x_2;A,C)-J(x_1;\tilde{A},\tilde{C})
			&=\mathbb{E}\left[e^{-q\tilde{\tau}}\mathbb{E}\left[\int_{0}^{\tau_1-\tilde{\tau}}e^{-qu}C(\tilde{\tau}+u)du\bigg|\mathcal{F}_{\tilde{\tau}}\right]\right]\\
			&\leq\mathbb{E}\left[\mathbb{E}\left[\int_{0}^{\tau_1-\tilde{\tau}}e^{-qu}C(\tilde{\tau}+u)du\bigg|\mathcal{F}_{\tilde{\tau}}\right]\right].
		\end{aligned}
	\end{equation}
	{From Theorem 2 of \citet{claisse2016}, we have
	\begin{equation*}
		\mathbb{E}\left[\int_{0}^{\tau_1-\tilde{\tau}}e^{-qu}C(\tilde{\tau}+u)du\bigg|\mathcal{F}_{\tilde{\tau}}\right]=J\left(x_2-x_1;\{A(t+\tilde{\tau})\}_{t\geq 0},\{C(t+\tilde{\tau})\}_{t\geq 0}\right),\quad \mathbb P-a.s.
	\end{equation*}
	By the inclusion property discussed in Remark \ref{remark on inclusion}, we have $(\{A(t+\tilde{\tau})\}_{t\geq 0},\{C(t+\tilde{\tau})\}_{t\geq 0})\in \Pi^{\mathscr{A},\mathscr{C}}_{x_2-x_1,a,c}$. Write $\overline a:=\argmax \mathscr A$ and $\underline c:=\argmin \mathscr C$. Then,}
	\begin{equation}\label{a4}
    \begin{aligned}
       J(x_2;A,C)-J(x_1;\tilde{A},\tilde{C})
       &\leq\mathbb{E}\left[\mathbb{E}\left[\int_{0}^{\tau_1-\tilde{\tau}}e^{-qu}C(\tilde{\tau}+u)du\bigg|\mathcal{F}_{\tilde{\tau}}\right]\right]\\
		&{=\mathbb{E}\left[J\left(x_2-x_1;\{A(t+\tilde{\tau})\}_{t\geq 0},\{C(t+\tilde{\tau})\}_{t\geq 0}\right)\right]}\\
        &{\leq \mathbb{E}\left[V\left(x_2-x_1;a,c)\right)\right]}\\
		&{\leq V(x_2-x_1,\overline a,\underline c).}
    \end{aligned}
		\end{equation}
	By Remark \ref{remark on no constraint case}, we then have
	\begin{equation*}
    \begin{aligned}
		V(x_2,a,c)-V(x_1,a,c)&\stackrel{\eqref{a3}}{\leq}J(x_2;A,C)-J(x_1;\tilde{A},\tilde{C})+\epsilon_1\\
		&{\stackrel{\eqref{a4}}{\leq}V(x_2-x_1,\overline a,\underline c)+\epsilon_1}\\
        &\leq V_{NC}(x_2-x_1)+\epsilon_1\\
        &\stackrel{\eqref{no constraint Lipschitz}}{\leq} L_0(x_2-x_1)+\epsilon_1.
    \end{aligned}
	\end{equation*}
	Since $\epsilon_1$ is arbitrary, then \eqref{a2} holds with $L_1=L_0$.
	
	(Part II) We now want to show that there exists $L_2>0$ such that
	\begin{equation}\label{a6}
		V(x,a,c_1)-V(x,a,c_2)\leq L_2\left(c_2-c_1\right)
	\end{equation}
	for all $c_1,c_2\in\mathscr{C}$ with $c_1\leq c_2$. Take $\epsilon_2>0$ and $\pi^2:=(A,C)\in\Pi^{\mathscr{A},\mathscr{C}}_{x,a,c_1}$ such that
	\begin{equation}\label{a7}
		J(x;A,C)\geq V(x,a,c_1)-\epsilon_2.
	\end{equation}
	Define the stopping time $\bar{T}=\min\{t:C(t)\geq c_2\}$ and denote $\tau_2$ the ruin time of the associated process $\{X_t^{\pi^2}\}_{t\geq 0}$. Consider $\bar{\pi}:=(A,\bar{C})\in\Pi^{\mathscr{A},\mathscr{C}}_{x,a,c_2}$ such that $\bar{C}(t)=c_2\mathbf{1}_{\{t<\bar{T}\}}+C(t)\mathbf{1}_{\{t\geq\bar{T}\}}$. Denote by $\{X_t^{\bar{\pi}}\}_{t\geq 0}$ the associated reserve process and by $\bar{\tau}\leq\tau_2$ the corresponding ruin time. Then,
	\begin{equation*}
		\bar{C}(t)-C(t)=
		\begin{cases}
			c_2-C(t),&\mbox{if $t<\bar T$},\\
			0,&\mbox{if $t\geq \bar{T}$}.
		\end{cases}
	\end{equation*}
	It then follows that $\bar{C}(t)-C(t)\leq c_2-c_1$ for any $t\geq 0$. Hence,
	\begin{equation}\label{a7.5}
    \begin{aligned}
		X_{\bar{\tau}}^{\pi^2}&=X_{\bar{\tau}}^{\pi^2}-X_{\bar{\tau}}^{\bar{\pi}}=\int_0^{\bar{\tau}}\left[\bar{C}(s)-C(s)\right]ds\leq\int_0^{\bar{\tau}}\left[c_2-c_1\right]ds=(c_2-c_1)\bar{\tau}.
        \end{aligned}
	\end{equation}
	Using similar arguments made in \eqref{a3.5} and \eqref{a4}, we then obtain 
	\begin{equation*}
    \begin{aligned}
		\mathbb{E}\left[\int_{\bar{\tau}}^{\tau_2}e^{-qs}C(s)ds\right]&=\mathbb{E}\left[\mathbb{E}\left[e^{-q\bar{\tau}}\int_0^{\tau_2-\bar{\tau}}e^{-qu}C(u+\bar{\tau})du\bigg|\mathcal{F}_{\bar{\tau}}\right]\right]\leq{\mathbb{E}\left[e^{-q\bar{\tau}}V(X^{\pi^2}_{\bar{\tau}},\overline a,\underline c)\right]}.
        \end{aligned}
	\end{equation*}
	Together with Remark \ref{remark on no constraint case}, we then have
	\begin{equation*}
    \begin{aligned}
		V(x,a,c_1)-V(x,a,c_2)&\stackrel{\eqref{a7}}{\leq}J(x;A,C)-J(x;A,\bar{C})+\epsilon_2\\
		&=\mathbb{E}\left[\int_0^{\bar{\tau}}e^{-qs}\left[C(s)-\bar C(s)\right]ds\right]+\mathbb{E}\left[\int_{\bar{\tau}}^{\tau_2}e^{-qs}C(s)ds\right]+\epsilon_2\\
		&\leq 0+\mathbb{E}\left[\int_{\bar{\tau}}^{\tau_2}e^{-qs}C(s)ds\right]+\epsilon_2\\
        &\leq{\mathbb{E}\left[e^{-q\bar{\tau}}V(X^{\pi^2}_{\bar{\tau}},\overline a,\underline c)\right]+\epsilon_2}\\
		&\leq\mathbb{E}\left[e^{-q\bar{\tau}}V_{NC}(X^{\pi^2}_{\bar{\tau}})\right]+\epsilon_2\\
        &\stackrel{\eqref{a7.5}}{\leq}\mathbb{E}\left[e^{-q\bar{\tau}}V_{NC}\left((c_2-c_1)\bar{\tau}\right)\right]+\epsilon_2\\
		&\stackrel{\eqref{no constraint Lipschitz}}{\leq}L_0(c_2-c_1)\mathbb{E}\left[e^{-q\bar{\tau}}\bar{\tau}\right]+\epsilon_2.
        \end{aligned}
	\end{equation*}
	Since $\epsilon_2$ is arbitrary, then choosing $L_2:=L_0\max_{t\geq 0}\{e^{-qt}t\}=\frac{L_0}{qe}$ yields \eqref{a6}.
	
	(Part III) Lastly, we want to show that there exists $L_3>0$ such that
	\begin{equation}\label{a6v2}
		V(x,a_1,c)-V(x,a_2,c)\leq L_3\left(a_1-a_2\right),
	\end{equation}
	for all $a_1,a_2\in\mathscr{A}$ with $a_2\leq a_1$. Take $\epsilon_3>0$ and $\pi^3:=(A,C)\in\Pi^{\mathscr{A}\mathscr{C}}_{x,a_1,c}$ such that
	\begin{equation}\label{a7v2}
		J(x;A,C)\geq V(x,a_1,c)-\epsilon.
	\end{equation}
	Define the stopping time $\hat{T}:=\min\{t:A(t)\leq a_2\}$ and denote $\tau_3$ the ruin time of the associated process $\{X_t^{\pi^3}\}_{t\geq 0}$. Consider $\hat{\pi}:=(\hat{A},C)\in\Pi^{\mathscr{A},\mathscr{C}}_{x,a_2,c}$ such that $\hat{A}(t)=a_2\mathbf{1}_{\{t<\hat{T}\}}+A(t)\mathbf{1}_{\{t\geq\bar{T}\}}$. Denote by $\{X_t^{\hat{\pi}}\}_{t\geq0}$ the associated reserve process and by $\hat{\tau}\leq \tau_3$ the corresponding ruin time. Then,
	\begin{equation*}
		A(t)-\hat{A}(t)=
		\begin{cases}
			A(t)-a_2,&\mbox{if $t<\hat{T}$},\\
			0,&\mbox{if $t\geq\hat{T}$}.
		\end{cases}
	\end{equation*}
	It then follows that ${A}(t)-\hat A(t)\leq a_1-a_2$ for any $t\geq 0$. Hence,
	\begin{equation}\label{a7.5v2}
    \begin{aligned}
		X_{\hat{\tau}}^{\pi^3}&=X_{\hat{\tau}}^{\pi^3}-X_{\hat{\tau}}^{\hat{\pi}}\\
		&=\int_0^{\hat{\tau}}\mu\left[A(s)-\hat{A}(s)\right]ds+\int_0^{\hat{\tau}}\sigma\left[A(s)-\hat A(s)\right]dW_s\\
		&{\leq\int_0^{\hat{\tau}}\mu\left[a_1-a_2\right]ds+\int_0^{\hat{\tau}}\sigma\left[A(s)-\hat A(s)\right]dW_s}\\
		&{=\mu(a_1-a_2)\hat{\tau}+\int_0^{\hat{\tau}}\sigma\left[A(s)-\hat A(s)\right]dW_s}.
        \end{aligned}
	\end{equation}
	Using similar arguments made in \eqref{a3.5} and \eqref{a4}, we then obtain
	\begin{equation*}
    \begin{aligned}
		\mathbb{E}\left[\int_{\hat{\tau}}^{\tau_3}e^{-qs}C(s)ds\right]&=\mathbb{E}\left[\mathbb{E}\left[e^{-q\hat{\tau}}\int_0^{\tau_3-\hat{\tau}}e^{-qu}C(u+\hat{\tau})du\bigg|\mathcal{F}_{\hat{\tau}}\right]\right]\leq{\mathbb{E}\left[e^{-q\hat{\tau}}V(X_{\hat{\tau}}^{\pi^3},\overline a,\underline c)\right]}.
        \end{aligned}
	\end{equation*}
	Together with Remark \ref{remark on no constraint case}, we then have
	\begin{equation*}
    \begin{aligned}
		V(x,a_1,c)-V(x,a_2,c)&\stackrel{\eqref{a7v2}}{\leq}J(x;A,C)-J(x;\hat{A},C)+\epsilon_3\\
		&=\mathbb{E}\left[\int_{\hat{\tau}}^{\tau_3}e^{-qs}C(s)ds\right]+\epsilon_3\\
		&{\leq \mathbb{E}\left[e^{-q\hat{\tau}}V(X_{\hat{\tau}}^{\pi^3},\overline a,\underline c)\right]+\epsilon_3}\\
		&\leq \mathbb{E}\left[e^{-q\hat{\tau}}V_{NC}(X^{\pi^3}_{\hat{\tau}})\right]+\epsilon_3\\
		&{\stackrel{\eqref{a7.5v2}}{\leq}\mathbb{E}\left[e^{-q\hat{\tau}}V_{NC}\left(\mu(a_1-a_2)\hat{\tau}+\int_0^{\hat{\tau}}\sigma\left[A(s)-\hat A(s)\right]dW_s\right)\right]+\epsilon_3}\\
		&{\stackrel{\eqref{no constraint Lipschitz}}{\leq}L_0(a_1-a_2)\mathbb{E}\left[e^{-q\hat{\tau}}\mu\hat{\tau}\right]+\sigma L_0\mathbb E \left[e^{-q\hat{\tau}}\int_0^{\hat{\tau}}\sigma\left[A(s)-\hat A(s)\right]dW_s\right]+\epsilon_3}\\
        &{\leq L_0(a_1-a_2)\mathbb{E}\left[e^{-q\hat{\tau}}\mu\hat{\tau}\right]+\sigma L_0\mathbb E \left[\int_0^{\hat{\tau}}\sigma\left[A(s)-\hat A(s)\right]dW_s\right]+\epsilon_3}\\
        &{= L_0(a_1-a_2)\mathbb{E}\left[e^{-q\hat{\tau}}\mu\hat{\tau}\right]+\epsilon_3}.
        \end{aligned}
	\end{equation*}
	{The term $\mathbb E \left[\int_0^{\hat{\tau}}\sigma\left[A(s)-\hat A(s)\right]dW_s\right]$ vanishes since the integrand is bounded above by $\sigma (a_1-a_2)$}. Since $\epsilon_3$ is arbitrary, then choosing $L_3:=L_0\mu\max_{t\geq 0}\{e^{-qt}t\}=\frac{L_0\mu}{qe}$ yields \eqref{a6v2}. Take $L=\max\{L_1,L_2,L_3\}$. The proof is complete.
\end{proof}

\begin{mypr}\label{claim on theta 1 and 2}
        {The following properties hold:
        \begin{itemize}
            \item[(i)] $\theta_2(a,c)<0<\theta_1(a,c)$,
            \item[(ii)] $\frac{\partial}{\partial c}\theta_1(a,c),\frac{\partial}{\partial c}\theta_2(a,c)>0$,
            \item[(iii)] $\frac{\partial}{\partial a}\theta_1(a,c)<0$,
            \item[(iv)] $\frac{\partial}{\partial a}\theta_2(a,c)< 0$ if $\mu a-2(b+c)< 0$ and $\sigma<\sqrt{\frac{-\mu(\mu a-2(b+c))}{2qa}}$,
            \item[(v)] $\frac{\partial}{\partial a}\theta_2(a,c)\geq 0$ if $\mu a-2(b+c)\geq 0$ \underline {or} if $\mu a-2(b+c)< 0$ and $\sigma\geq \sqrt{\frac{-\mu(\mu a-2(b+c))}{2qa}}$. Equality holds if and only if $\mu a-2(b+c)< 0$ and $\sigma=\sqrt{\frac{-\mu(\mu a-2(b+c))}{2qa}}$.
        \end{itemize}}
	 \end{mypr}
\begin{proof}
	Since $\sqrt{(b+c-\mu a)^2+2q\sigma^2a^2}>\sqrt{(b+c-\mu a)^2}\geq b+c-\mu a$, then the first result follows. 
	
	Taking the partial derivative with respect to $c$, we have
	\begin{align*}
		\frac{\partial}{\partial c}\theta_1(a,c)=\frac{1}{\sigma^2a^2}\left[1+\frac{b+c-\mu a}{\sqrt{(b+c-\mu a)^2+2q\sigma^2a^2}}\right]
	\end{align*}
	and
	\begin{align*}
		\frac{\partial}{\partial c}\theta_2(a,c)=\frac{1}{\sigma^2a^2}\left[1-\frac{b+c-\mu a}{\sqrt{(b+c-\mu a)^2+2q\sigma^2a^2}}\right].
	\end{align*}
	Since $\frac{b+c-\mu a}{\sqrt{(b+c-\mu a)^2+2q\sigma^2a^2}}<1$, then the second result follows. 
	
	Write $\ol b:=b+c$. Taking the partial derivative of $\theta_2(a,c)$ with respect to $a$, we have
	\begin{align*}
		\frac{\partial}{\partial a}\theta_2(a,c)=\frac{1}{\sigma^2a^3}\left[\mu a-2\ol b+\frac{(\mu a-\ol b)(\mu a-2\ol b)+2q\sigma^2a^2}{\sqrt{(\ol b-\mu a)^2+2q\sigma^2a^2}}\right].
	\end{align*}
	If $\mu a-2\ol b\geq 0$, then $\mu a-\ol b>0$. Consequently, we have $\frac{\partial}{\partial a}\theta_2(a,c)>0$, which is the first ``if" of the fifth result. 
	
	We now consider the case where $\mu a-2\ol b< 0$ and  $\sigma<\sqrt{\frac{-\mu(\mu a-2\ol b)}{2qa}}$. Then,
	\begin{equation}\label{eqn:a4}
		\begin{aligned}
			0&>2q\sigma^2a^2+\mu a(\mu a-2\ol b)\\
			&=4q^2\sigma^4a^4+4q\sigma^2a^2(\mu a-\ol b)(\mu a -2\ol b)+(\mu a-\ol b)^2(\mu a-2\ol b)^2-(\mu a-2\ol b)^2\left[(\mu a-\ol b)^2+2q\sigma^2a^2\right].
		\end{aligned}
	\end{equation}
	Hence,
    \begin{align*}
		\mu a-2\ol b+\frac{(\mu a-\ol b)(\mu a-2\ol b)+2q\sigma^2a^2}{\sqrt{(\ol b-\mu a)^2+2q\sigma^2a^2}}<0,
	\end{align*}
	which proves the fourth result.
	
	Suppose $\mu a-2\ol b< 0$ and $\sigma\geq \sqrt{\frac{-\mu(\mu a-2\ol b)}{2qa}}$. Similar to the previous case, we have
	\begin{equation}\label{eqn:a1}
		\begin{aligned}
			0&\leq 2q\sigma^2a^2+\mu a(\mu a-2\ol b)\\
			&=4q^2\sigma^4a^4+4q\sigma^2a^2(\mu a-\ol b)(\mu a -2\ol b)+(\mu a-\ol b)^2(\mu a-2\ol b)^2-(\mu a-2\ol b)^2\left[(\mu a-\ol b)^2+2q\sigma^2a^2\right].
		\end{aligned}
	\end{equation}
	Hence,
	\begin{align}\label{eqn:a2}
		|\mu a-2\ol b|\leq \left|\frac{(\mu a-\ol b)(\mu a-2\ol b)+2q\sigma^2a^2}{\sqrt{(\ol b-\mu a)^2+2q\sigma^2a^2}}\right|.
	\end{align}
	From the first line of \eqref{eqn:a1}, we have that $0\leq 2q\sigma^2a^2+\mu a(\mu a-2\ol b)$. Then, 
	\begin{align}\label{eqn:a3}
		0< 2q\sigma^2a^2+\mu a(\mu a-2\ol b)-\ol b(\mu a-2\ol b)=(\mu a-\ol b)(\mu a-2\ol b)+2q\sigma^2a^2.
	\end{align}
	Combining \eqref{eqn:a2} and \eqref{eqn:a3} yields
	\begin{align*}
		\mu a-2\ol b+\frac{(\mu a-\ol b)(\mu a-2\ol b)+2q\sigma^2a^2}{\sqrt{(\ol b-\mu a)^2+2q\sigma^2a^2}}\geq 0,
	\end{align*}
	which proves the second ``if" of the fifth result. 
	
	Taking the partial derivative of $\theta_1(a,c)$ with respect to $a$, we have
	\begin{align*}
		\frac{\partial}{\partial a}\theta_1(a,c)=-\frac{1}{\sigma^2a^3}\left[2\ol b-\mu a+\frac{(\mu a-\ol b)(\mu a-2\ol b)+2q\sigma^2a^2}{\sqrt{(\ol b-\mu a)^2+2q\sigma^2a^2}}\right].
	\end{align*}
	We want to prove that $\frac{\partial}{\partial a}\theta_1(a,c)<0$. Suppose otherwise, that is,
	\begin{align*}
		2\ol b-\mu a+\frac{(\mu a-\ol b)(\mu a-2\ol b)+2q\sigma^2a^2}{\sqrt{(\ol b-\mu a)^2+2q\sigma^2a^2}}\leq 0.
	\end{align*}
	Write $d:=\frac{(\mu a-\ol b)(\mu a-2\ol b)+2q\sigma^2a^2}{\sqrt{(\ol b-\mu a)^2+2q\sigma^2a^2}}$. The inequality holds if one of the following cases is true: (1) $2\ol b-\mu a\leq 0$ and $d\geq 0$ with $(2\ol b-\mu a)^2\geq d^2$, (2) $2\ol b-\mu a\geq 0$ and $d\leq 0$ with $(2\ol b-\mu a)^2\leq d^2$, or (3) $2\ol b-\mu a\leq 0$ and $d\leq 0$. For case (1), we obtain in the reverse direction of \eqref{eqn:a4} that $\sigma^2 <\frac{-\mu(\mu a-2\ol b)}{2qa}<0$, which is a contradiction. For case (2), it must be the case that $\mu a-\ol b>0$ since $\mu a-\ol b\leq 0$ will contradict the premise that $d\leq 0$. We obtain \eqref{eqn:a1} in a reverse direction. From the first line of \eqref{eqn:a1}, we have $0\leq 2q\sigma^2a^2+\mu a(\mu a-2\ol b)<2q\sigma^2a^2+(\mu a-\ol b)(\mu a-2\ol b)$, which leads to a contradiction since it must be the case that $d\leq 0$. For case (3), since $\mu a-2\ol b>0$, then $\mu a-\ol b>0$. This is also a contradiction since it must be the case that $d\leq 0$. Therefore, we obtain the third result, which completes the proof. 
\end{proof}

	The next result is a comparison principle for the finite case and is used to prove the uniqueness of the viscosity solution to the HJB equation \eqref{hjb main}. 
	
	\begin{mypr}\label{comparison viscosity}
		Suppose the following conditions hold:
		\begin{itemize}
			\item[(i)] $\underline{u}$ is a viscosity subsolution and $\overline{u}$ is a viscosity supersolution of the HJB equation \eqref{hjb main} for all $x>0$,
			\item[(ii)] $\ul{u}$ and $\ol{u}$ are nondecreasing in the variable $x$ and Lipschitz in $[0,\infty)$, and
			\item[(iii)] $\ul{u}(0)=\ol{u}(0)$ and $\lim_{x\to\infty}\ul{u}(x)\leq\frac{c_n}{q}\leq \lim_{x\to\infty}\ol{u}(x)$.
		\end{itemize}
		Then, $\ul{u}\leq\ol{u}$ in $[0,\infty)$.
	\end{mypr}
    \begin{proof}[Proof of Proposition \ref{comparison viscosity}]
	Suppose that there is a point $x_0\in[0,\infty)$ such that
	\begin{equation*}
		\ul{u}(x_0)-\ol{u}(x_0)>0.
	\end{equation*} 
	For every $\gamma>1$, define $h_j:=1+\eta e^{-\frac{c_j}{c_n}}$ and $\ol{u}^{\gamma}(x)=\gamma h_j\ol{u}(x)$, where
	\begin{equation*}
		\eta=\frac{\ul{u}(x_0)-\ol{u}(x_0)}{2\ol{u}(x_0)}>0.
	\end{equation*}
	Then, $\varphi$ is a test function for supersolution of $\ol{u}$ at $x$ if and only if $\varphi^{\gamma}:=\gamma h_j\varphi$ is a test function for supersolution of $\ol{u}^{\gamma}$ at $x$. Using \eqref{hjb supersolution} and the fact that $1-\gamma h_j<1-\gamma<0$,
	\begin{equation*}
    \begin{aligned}
		\mathcal{L}^{a_i,c_j}(\varphi^{\gamma})(x)&=\frac{\sigma^2a_i^2}{2}\gamma h_j\varphi_{xx}(x)+(\mu a_i-b-c_j)\gamma h_j\varphi_x(x)-q\gamma h_j\varphi(x)+c_j\\
		&=\gamma h_j\mathcal{L}^{a_i,c_j}(\varphi)(x)+c(1-\gamma h_j)<0,
        \end{aligned}
	\end{equation*}
	and, provided $\varphi(x)>0$,
	\begin{equation*}
    \begin{aligned}
		V^{a_{i+1},c_j}(x)-\varphi^{\gamma}(x)=V^{a_{i+1},c_j}(x)-\gamma h_j\varphi(x)< V^{a_{i+1},c_j}(x)-\varphi(x)\leq 0,\\
		V^{a_{i},c_{j+1}}(x)-\varphi^{\gamma}(x)=V^{a_{i},c_{j+1}}(x)-\gamma h_j\varphi(x)< V^{a_{i},c_{j+1}}(x)-\varphi(x)\leq 0.
        \end{aligned}
	\end{equation*}
	
	Take $\gamma_0>1$ such that
	\begin{equation}\label{a17.5}
		\ul{u}(x_0)-\ol{u}^{\gamma_0}(x_0)>0.
	\end{equation}
	Define
	\begin{equation}\label{a18}
		M:=\sup_{x\geq 0}\left[\ul{u}(x)-\ol{u}^{\gamma_0}(x)\right].
	\end{equation}
	Since $\lim_{x\to\infty}\ul{u}(x)\leq\frac{c_n}{q}\leq \lim_{x\to\infty}\ol{u}(x)$, there exists $\tilde{x}>x_0$ such that
	\begin{equation}\label{a19}
		\ul{u}(x)-\ol{u}^{\gamma_0}(x)\leq 0\quad\mbox{for $x\geq \tilde{x}$}.
	\end{equation}
	We then have
	\begin{equation*}
		0\stackrel{\eqref{a17.5}}{<}\ul{u}(x_0)-\ol{u}^{\gamma_0}(x_0)\stackrel{\eqref{a18}}{\leq} M=\max_{x\in[0,\tilde{x}]}\left[\ul{u}(x)-\ol{u}^{\gamma_0}(x)\right].
	\end{equation*}
	
	Define {$x^*:=\argmax_{x\in[0,\tilde x]}\left[\ul{u}(x)-\ol{u}^{\gamma_0}(x)\right]$ }. Consider the set
	\begin{equation*}
		\mathcal{S}:=\left\{(x,y):0\leq x\leq y\leq \tilde{x}\right\}
	\end{equation*}
	and, for all $\xi>0$, the functions
	\begin{equation}\label{a21}
		\begin{aligned}
			\Phi^{\xi}(x,y)&=\frac{\xi}{2}(x-y)^2+\frac{2m}{\xi^2(y-x)+\xi},\\
			\Sigma^{\xi}(x,y)&=\ul{u}(x)-\ol{u}^{\gamma_0}(y)-\Phi^{\xi}(x,y),
		\end{aligned}
	\end{equation}
	where $m>0$ is a Lipschitz constant satisfying
	\begin{equation}\label{a19.5}
		\begin{aligned}
			\left|\ul{u}(x)-\ul{u}(y)\right|&\leq m|x-y|,\\
			\left|\ol{u}^{\gamma_0}(x)-\ol{u}^{\gamma_0}(y)\right|&\leq m|x-y|.
		\end{aligned}
	\end{equation}
	The partial derivatives of $\Phi^{\xi}$ satisfy
	\begin{equation}\label{a21.5}
		\Phi^{\xi}_x(x,y)=\xi(x-y)+\frac{2m}{(\xi(y-x)+1)^2}=-\Phi^{\xi}_y(x,y).
	\end{equation}

	Define $M^{\xi}=\max_{\mathcal{S}}\Sigma^{\xi}$ and $(x_{\xi},y_{\xi})=\argmax_{\mathcal{S}}\Sigma^{\xi}$. We then obtain
	\begin{equation*}
		M^{\xi}\geq \Sigma^{\xi}(x^*,x^*)=M-\Phi^{\xi}(x^*,x^*)=M-\frac{2m}{\xi}.
	\end{equation*}
	Hence,
	\begin{equation}\label{a22}
		\liminf_{\xi\to\infty}M^{\xi}\geq M.
	\end{equation}
	
	It can be shown that the maximum is not achieved on the boundary $y=x$, and similar arguments apply to the boundaries $x=0$ and $y=\tilde{x}$. Thus, there exists $\xi_0$ large enough such that if $\xi\geq\xi_0$, then $(x_{\xi},y_{\xi})\notin\partial\mathcal{S}$.
	
	Using the inequality $\Sigma^{\xi}(x_{\xi},x_{\xi})+\Sigma^{\xi}(y_{\xi},y_{\xi})\leq 2\Sigma^{\xi}(x_{\xi},y_{\xi})$ and \eqref{a19.5}, we obtain
	\begin{equation}\label{a23}
		\xi|x_{\xi}-y_{\xi}|^2\leq 6m|x_{\xi}-y_{\xi}|.
	\end{equation}
	We can then find a sequence $\xi_{n}\to\infty$ such that $(x_{\xi_n},y_{\xi_n})\to(\hat{x},\hat{y})\in\mathcal{S}$. From \eqref{a23}, we get
	\begin{equation}\label{a24}
		|x_{\xi_n}-y_{\xi_n}|\leq \frac{6m}{\xi_{n}},
	\end{equation}
	which gives $\hat{x}=\hat{y}$ as $n\to\infty$.
	
	Since $\Sigma^{\xi}$ reaches the maximum in $(x_{\xi},y_{\xi})$ in the interior of the set $\mathcal{S}$, then 
	\begin{equation*}
		0\leq \Sigma^{\xi}(x_{\xi},y_{\xi})-\Sigma^{\xi}(x,y_{\xi})=\ul{u}(x_{\xi})-\ul{u}(x)-\Phi^{\xi}(x_{\xi},y_{\xi})+\Phi^{\xi}(x,y_{\xi}).
	\end{equation*}
	Hence, the function $\psi(x):=\Phi^{\xi}(x,y_{\xi})-\Phi^{\xi}(x_{\xi},y_{\xi})+\ul{u}(x_{\xi})$ is a test for subsolution for $\ul{u}$ at $x_{\xi}$, and so,
	\begin{equation*}
		\max\{\mathcal{L}^{a_i,c_j}(\psi)(x_{\xi}),V^{a_{i+1},c_{j}}(x_{\xi})-\psi(x_{\xi}),V^{a_i,c_{j+1}}(x_{\xi})-\psi(x_{\xi})\}\geq 0.
	\end{equation*}
	Similarly, the function $\varphi(x):=-\Phi^{\xi}(x_{\xi},y)+\Phi^{\xi}(x_{\xi},y_{\xi})+\ol{u}^{\gamma_0}(y_{\xi})$ is a test for supersolution for $\ol{u}^{\gamma_0}$ at $y_{\xi}$, and so,

	\begin{equation*}
		\max\{\mathcal{L}^{a_i,c_j}(\varphi)(y_{\xi}),V^{a_{i+1},c_{j}}(y_{\xi})-\varphi(y_{\xi}),V^{a_i,c_{j+1}}(y_{\xi})-\varphi(y_{\xi})\}\leq 0.
	\end{equation*}
	
	Assume first that the functions $\ul{u}(x)$ and $\ol{u}^{\gamma_0}(y)$ are twice continuously differentiable at $x_{\xi}$ and $y_{\xi}$, respectively. Since $\Sigma^{\xi}$ defined in \eqref{a21} reaches a local maximum at $(x_{\xi},y_{\xi})\notin\partial\mathcal{S}$, by the classical maximum principle, we have $\Sigma_x^{\xi}(x_{\xi},y_{\xi})=\Sigma_y^{\xi}(x_{\xi},y_{\xi})=0$ and
	\begin{equation}\label{a50}
		H(\Sigma^{\xi})(x_{\xi},y_{\xi}):=\begin{bmatrix}
			A-\Phi^{\xi}_{xx}(x_{\xi},y_{\xi})&-\Phi^{\xi}_{xy}(x_{\xi},y_{\xi})\\
			-\Phi^{\xi}_{xy}(x_{\xi},y_{\xi})&-B-\Phi^{\xi}_{yy}(x_{\xi},y_{\xi})
		\end{bmatrix}\preceq 0,
	\end{equation}
    where $A=\ul{u}_{xx}(x_{\xi})$, $B=\ol{u}^{\gamma_0}_{yy}(y_{\xi})$, and $D\preceq 0$ means $D$ is a negative semi-definite matrix. We also write $D_1\preceq D_2$ to mean that $D_1-D_2$ is a negative semi-definite matrix. Then, \eqref{a50} can be rewritten as
	\begin{equation*}
    \begin{aligned}
		\begin{bmatrix}
			A&0\\
			0&-B
		\end{bmatrix}&\preceq H(\Phi^{\xi})(x_{\xi},y_{\xi})\\
        &:=
		\begin{bmatrix}
			\Phi^{\xi}_{xx}(x_{\xi},y_{\xi})&\Phi^{\xi}_{xy}(x_{\xi},y_{\xi})\\
			\Phi^{\xi}_{xy}(x_{\xi},y_{\xi})&\Phi^{\xi}_{yy}(x_{\xi},y_{\xi})
		\end{bmatrix}\\
  &\stackrel{\eqref{a21.5}}{=}\left(\xi+\frac{4m\xi}{(\xi(y_{\xi}-x_{\xi})+1)^3}\right)
		\begin{bmatrix}
			1&-1\\
			-1&1
		\end{bmatrix}.
        \end{aligned}
	\end{equation*}
    Then,
	\begin{equation*}
		\ul{u}_x(x_{\xi})=\Phi^{\xi}_x(x_{\xi},y_{\xi})=\xi(x_{\xi}-y_{\xi})+\frac{2m}{(\xi(y_{\xi}-x_{\xi})+1)^2}
		=-\Phi^{\xi}_y(x_{\xi},y_{\xi})=\ol{u}^{\gamma_0}_y(y_{\xi}).
	\end{equation*}
    Moreover, since $H(\Sigma^{\xi})(x_{\xi},y_{\xi})$ is negative semi-definite,
	\begin{equation*}
    \begin{aligned}
		0&\geq 
		\begin{bmatrix}
			1&1
		\end{bmatrix}
		\begin{bmatrix}
			A&0\\
			0&-B
		\end{bmatrix}
		\begin{bmatrix}
			1\\
			1
		\end{bmatrix}-
		\left(\xi+\frac{4m\xi}{(\xi(y_{\xi}-x_{\xi})+1)^3}\right)
		\begin{bmatrix}
			1&1
		\end{bmatrix}
		\begin{bmatrix}
			1&-1\\
			-1&1
		\end{bmatrix}
		\begin{bmatrix}
			1\\
			1
		\end{bmatrix}=A-B.
        \end{aligned}
	\end{equation*}
	In the case that $\ul{u}(x)$ and $\ol{u}^{\gamma_0}(y)$ are not twice continuously differentiable at $x_{\xi}$ and $y_{\xi}$, respectively, we will use Theorem 3.2 of \citet{crandall1992}. Since the assumptions of the theorem are satisfied, then, for any $\delta>0$, there exist real numbers $A_{\delta}$ and $B_{\delta}$ such that
	\begin{equation*}
		\begin{bmatrix}
			A_{\delta}&0\\
			0&-B_{\delta}
		\end{bmatrix}
		\preceq H(\Phi^{\xi})(x_{\xi},y_{\xi})+\delta\left[H(\Phi^{\xi})(x_{\xi},y_{\xi})\right]^2
	\end{equation*}
	and
	\begin{equation}\label{a27}
		\begin{aligned}
			\frac{\sigma^2a_i^2}{2}A_{\delta}+(\mu a_i-b-c_j)\psi_x(x_{\xi})-q\psi(x_{\xi})+c_j&\geq 0\\
			\frac{\sigma^2a_i^2}{2}B_{\delta}+(\mu a_i-b-c_j)\varphi^{\gamma_0}_y(y_{\xi})-q\varphi^{\gamma_0}(y_{\xi})+c_j&\leq 0.
		\end{aligned}
	\end{equation}
	Then,
	\begin{equation}\label{a27.5}
    \begin{aligned}
		0&\geq
		\begin{bmatrix}
			1&1	
		\end{bmatrix}
		\begin{bmatrix}
			A_{\delta}&0\\
			0&-B_{\delta}
		\end{bmatrix}
		\begin{bmatrix}
			1\\
			1
		\end{bmatrix}-
		\begin{bmatrix}
			1&1	
		\end{bmatrix}
		\left[H(\Phi^{\xi})(x_{\xi},y_{\xi})+\delta\left[H(\Phi^{\xi})(x_{\xi},y_{\xi})\right]^2\right]
		\begin{bmatrix}
			1\\
			1
		\end{bmatrix}\\
		&=(A_{\delta}-B_{\delta})\\
  &\quad-\left(\xi+\frac{4m\xi}{(\xi(y_{\xi}-x_{\xi})+1)^3}+2\delta\left(\xi+\frac{4m\xi}{(\xi(y_{\xi}-x_{\xi})+1)^3}\right)^2\right)
		\begin{bmatrix}
			1&1	
		\end{bmatrix}
		\begin{bmatrix}
			1&-1\\
			-1&1
		\end{bmatrix}
		\begin{bmatrix}
			1\\
			1
		\end{bmatrix}\\
		&=A_{\delta}-B_{\delta}.
        \end{aligned}
	\end{equation}
	Since $\varphi^{\gamma_0}(y_{\xi})=\ul{u}^{\gamma_0}(y_{\xi})$, $\psi(x_{\xi})=\ul{u}(x_{\xi})$, and
   \begin{equation*}
		\varphi^{\gamma_0}_y(y_{\xi})=-\Phi^{\xi}_y(x_{\xi},y_{\xi})=\Phi^{\xi}_x(x_{\xi},y_{\xi})=\psi_x(x_{\xi}),
	\end{equation*}
  we get
	\begin{equation}\label{a28}
		\ul{u}(x_{\xi})-\ol{u}^{\gamma_0}(y_{\xi})=\psi(x_{\xi})-\varphi^{\gamma_0}(y_{\xi})\stackrel{\eqref{a27}}{\leq} \frac{\sigma^2a_i^2}{2q}(A_{\delta}-B_{\delta})\stackrel{\eqref{a27.5}}{\leq} 0.
	\end{equation}
   Hence,
	\begin{equation*}
    \begin{aligned}
		0<M&\stackrel{\eqref{a22}}{\leq} \liminf_{\xi\to\infty}M^{\xi}\leq \lim_{n\to\infty}M^{\xi_n}=\lim_{n\to\infty}\Sigma^{\xi_n}(x_{\xi_n},y_{\xi_n})\\
		&=\lim_{n\to\infty}\left[\ul{u}(x_{\xi_n})-\ol{u}^{\gamma_0}(y_{\xi_n})-\frac{\xi_n}{2}(x_{\xi_n}-y_{\xi_n})^2-\frac{2m}{\xi_n^2(y_{\xi_n}-x_{\xi_n})+\xi_n}\right]\\
		&\stackrel{\eqref{a24}}{=}\ul{u}(\hat{x})-\ol{u}^{\gamma_0}(\hat{x})\stackrel{\eqref{a28}}{\leq} 0,
        \end{aligned}
	\end{equation*}
    which is a contradiction. The proof is complete.
\end{proof}

    \begin{mylm}\label{G has maxima and min exists}
		The functions $G^{\mathcal{A}}_{i,j}$, $G^{\mathcal{C}}_{i,j}$, and $G^{\mathcal{E}}_{i,j}$ attain their respective maxima in $[0,\infty)$. Moreover, $\min\left[\argmax_{x\in[0,\infty)}G^{\mathcal{A}}_{i,j}(x)\right]$, $\min\left[\argmax_{x\in[0,\infty)}G^{\mathcal{C}}_{i,j}(x)\right]$, and $\min\left[\argmax_{x\in[0,\infty)}G^{\mathcal{E}}_{i,j}(x)\right]$ exist.
	\end{mylm}
    \begin{proof}
	Letting $x\to\infty$, we have the following inequality
    \begin{equation*}
		\lim_{x\to\infty}\wopt(x,a_i,c_j)=\frac{c_n}{q}>\frac{c_j}{q}.
	\end{equation*}
    Since $\theta_1(a_i,c_j)>0>\theta_2(a_i,c_j)$, then, for large enough $x$,
	\begin{equation}\label{G is positive}
		G^{\mathcal{A}}_{i,j}(x)=\frac{\frac{c_n}{q}-\frac{c_j}{q}\left(1-e^{\theta_2(a_i,c_j)x}\right)}{e^{\theta_1(a_i,c_j){x}}-e^{\theta_2(a_i,c_j){x}}}=\frac{\frac{c_n}{q}-\frac{c_j}{q}+\frac{c_j}{q}e^{\theta_2(a_i,c_j)x}}{e^{\theta_1(a_i,c_j){x}}-e^{\theta_2(a_i,c_j){x}}}>0.
	\end{equation}
    Moreover,
	\begin{equation}\label{G tends to 0}
		\lim_{x\to\infty}G^{\mathcal{A}}_{i,j}(x)=\frac{\frac{c_n}{q}}{\lim_{x\to\infty}\left(e^{\theta_1(a_i,c_j){x}}-e^{\theta_2(a_i,c_j){x}}\right)}=0.
	\end{equation}
    Combining \eqref{G is positive} and \eqref{G tends to 0} with the definition of $G^{\mathcal{A}}_{i,j}(0)$ and the continuity of $G^{\mathcal{A}}_{i,j}$ on $[0,\infty)$ implies that $G^{\mathcal{A}}_{i,j}$ attains its maximum in $[0,\infty)$. Since $G^{\mathcal{A}}_{i,j}$ is continuous on $[0,\infty)$, then the set $\argmax_{\tilde{x}\in[0,\infty)}G^{\mathcal{A}}_{i,j}(\tilde{x})$ is closed. Hence, its minimum exists. The same analysis holds for $G^{\mathcal{C}}_{i,j}$ and $G^{\mathcal{E}}_{i,j}$.
\end{proof}

    \begin{mylm}\label{lemma on existence of U}
		$U^*_{\mathcal{A}}$, $U^*_{\mathcal{C}}$, and $U^*_{\mathcal{E}}$ exist.
	\end{mylm}
    \begin{proof}
	From \eqref{ode solution base}, the solutions $U$ of the equation $\mathcal{L}^{a_i,c_j}(U)=0$ in $[0,\infty)$ with boundary condition $U(0)=0$ are of the form
	\begin{equation*}
		U_{k^{i,j}_{\mathcal{A}}}(x)=\frac{c_j}{q}\left[1-e^{\theta_2(a_i,c_j)x}\right]+k^{i,j}_{\mathcal{A}}\left[e^{\theta_1(a_i,c_j)x}-e^{\theta_2(a_i,c_j)x}\right].
	\end{equation*}
    Let $k^{i,j}_{\mathcal{A}}\geq 0$. From \eqref{G^A function}, $U_{k^{i,j}_{\mathcal{A}}}(x)\geq \wopt(x,a_{i+1},c_j)$ for all $x\geq 0$ if and only if $k^{i,j}_{\mathcal{A}}\geq G_{i,j}^{\mathcal{A}}(x)$ for all $x\geq 0$. Using Lemma \ref{G has maxima and min exists}, there exists $\hat{k}^{i,j}_{\mathcal{A}}=\max_{x\in[0,\infty)}G^{\mathcal{A}}_{i,j}(x)>0$. Hence, $U^*_{\mathcal{A}}=U_{\hat{k}^{i,j}_{\mathcal{A}}}$. The proof for $U^*_{\mathcal{C}}$, and $U^*_{\mathcal{E}}$ is similar.
\end{proof}

    The following lemma establishes the differentiability and continuity of $W^{y^*,z^*}$.
		
	\begin{mylm}\label{W is cont diff}
		$\wopt(x,a_i,c_j)$ is infinitely continuously differentiable at all $x\in[0,\infty)\setminus\{y^*_{k,j}:k=i,\ldots,m-1\}\cup\{z^*(c_{i,l}):l=j,\ldots,n-1\}$ and is continuously differentiable at the points $y^*_{k,j}$ and $z^*_{i,l}$ for $k=i,\ldots,m-1$ and $l=j,\ldots,n-1$.
	\end{mylm}
    \begin{proof}
        The result follows directly by construction and a recursive argument. 
    \end{proof}

	The next lemma proves an inequality for $\wopt_{xx}$ at a neighborhood of threshold values. The result can be used to prove that $\wopt$ is a viscosity solution since $\wopt_{xx}$ may not exist as stated in Lemma \ref{W is cont diff}.
	
	\begin{mylm}\label{Wxx inequality}
		For $1\leq i\leq m-1$ and $1\leq j\leq n-1$, we have $\wopt_{xx}(y^{*-}_{i,j},a_i,c_j)-\wopt_{xx}(y^{*+}_{i,j},a_i,c_j)\geq 0$ and $\wopt_{xx}(z^{*-}_{i,j},a_i,c_j)-\wopt_{xx}(z^{*+}_{i,j},a_i,c_j)\geq 0.$
    \end{mylm}
    \begin{proof}
	We prove the first inequality. By definition, it must hold that $U_{\mathcal{A}}^*(\cdot)-W^{y^*,z^*}(\cdot,a_{i+1},c_j)$ reaches the minimum at $y^*_{i,j}$. It then follows that $\frac{d}{dx}U_{\mathcal{A}}^*(y^*_{i,j})-W_x^{y^*,z^*}(y^*_{i,j},a_{i+1},c_j)=0$. Moreover, it holds that
	\begin{equation*}
		W^{y^*,z^*}(x,a_i,c_j)=U^*_{\mathcal{A}}(x)\mathbf{1}_{\{x<y_{i,j}^*\}}+W^{y^*,z^*}(x,a_{i+1},c_j)\mathbf{1}_{\{x\geq y_{i,j}^*\}}.
	\end{equation*}
    Hence,
	\begin{equation*}
		W_{xx}^{y^*,z^*}(y^*_{i,j}-,a_i,c_j)-W_{xx}^{y^*,z^*}(y^*_{i,j}+,a_i,c_j)=\frac{d^2}{dx^2}U^*_{\mathcal{A}}(y_{i,j}^*)-W_{xx}^{y^*,z^*}(y^*_{i,j}+,a_{i+1},c_j)\geq 0.
	\end{equation*}
    The proof for the second inequality is similar.
\end{proof}

\section{Proofs of the Results in Section \ref{section:hjb}}\label{appendix: b}

\begin{proof}[Proof of Theorem \ref{verification discrete}]
    (Part I) We first show that $V^{a_i,c_j}$ is a viscosity supersolution.  {By Proposition \ref{bound and monotone}, $V^{a_{i+1},c_j}(x)-V^{a_i,c_j}(x)\leq 0$ and $V^{a_i,c_{j+1}}(x)-V^{a_i,c_j}(x)\leq 0$ in $(0,\infty)$.} 
	
	Consider an $x\in(0,\infty)$ and the admissible strategy $\pi:=(A,C)\in\Pi^{\mathscr{A},\mathscr{C}}_{x,a_i,c_j}$, which retains incoming claims at a constant rate $a_i$ and pays dividends at a constant rate $c_j$ up to the ruin time $\tau_{\pi}$. Let $\{X^{\pi}_t\}_{t\geq 0}$ be the corresponding reserve process. Suppose there exists a test function $\varphi$ for supersolution of \eqref{hjb main} at $x$. Then $\varphi\leq V^{a_i,c_j}$ and $\varphi(x)=V^{a_i,c_j}(x)$. 
	
	We want to prove that $\mathcal{L}^{a_i,c_j}(\varphi)(x)\leq 0$. Since $\mathcal{L}^{a_i,c_j}(\varphi)(\cdot)$ may be unbounded, we consider an auxiliary function $\tilde{\varphi}$ for the supersolution of \eqref{hjb main} such that $\tilde{\varphi}\leq \varphi\leq V^{a_i,c_j}$ in $[0,\infty)$, $\tilde{\varphi}=\varphi$ in $[0,2x]$ and $\mathcal{L}^{a_i,c_j}(\tilde{\varphi})(\cdot)$ is bounded in $[0,\infty)$. We construct $\tilde{\varphi}$ by considering a function $g:[0,\infty)\to[0,1]$ such that $g=0$ in $[2x+1,\infty)$ and $g=1$ in $[0,2x]$. Define $\tilde{\varphi}(w)=\varphi(w)g(w)$. Using Lemma \ref{dpp}, we obtain for $h>0$
	\begin{equation*}
		\tilde{\varphi}(x)=V^{a_i,c_j}(x)\geq \mathbb{E}\left[\int_0^{\tau_{\pi}\wedge h}e^{-qs}c_j\phantom{.}ds\right]+\mathbb{E}\left[e^{-q(\tau_{\pi}\wedge h)}\tilde{\varphi}\left(X_{\tau_{\pi}\wedge h}^{\pi}\right)\right].
	\end{equation*}
	Using Itô's formula and \eqref{HJB base} yields
	\begin{equation}\label{inequality expectation}
    \begin{aligned}
		0&\geq\mathbb{E}\left[\int_0^{\tau_{\pi}\wedge h}e^{-qs}c_j\phantom{.}ds\right]+\mathbb{E}\left[e^{-q(\tau_{\pi}\wedge h)}\tilde{\varphi}\left(X_{\tau_{\pi}\wedge h}^{\pi}\right)-\tilde{\varphi}(x)\right]\\
		&=\mathbb{E}\left[\int_0^{\tau_{\pi}\wedge h}e^{-qs}c_j\phantom{.}ds\right]+{\mathbb{E}\left[\int_0^{\tau_{\pi}\wedge h}e^{-qs}\tilde{\varphi}_x(X_s^{\pi})\sigma dW_s\right]}\\
  &\quad+\mathbb{E}\left[\int_0^{\tau_{\pi}\wedge h}e^{-qs}\left((\mu a_i-b-c_j)\tilde{\varphi}_x(X_s^{\pi})+\frac{\sigma^2a_i^2}{2}\tilde{\varphi}_{xx}(X_s^{\pi})-q\tilde{\varphi}(X_s^{\pi})\right)ds\right]\\
		&=\mathbb{E}\left[\int_0^{\tau_{\pi}\wedge h}e^{-qs}\mathcal{L}^{a_i,c_j}(\tilde{\varphi})({X_s^{\pi}})ds\right].
    \end{aligned}
	\end{equation}
	{The term $\mathbb{E}\left[\int_0^{\tau_{\pi}\wedge h}e^{-qs}\tilde{\varphi}_x(X_s^{\pi})\sigma dW_s\right]$ vanishes to 0 because the integrand is bounded. Since $\varphi$ is continuous and differentiable, it is locally bounded, which implies that it is locally Lipschitz. Consequently, $\varphi_x=\tilde \varphi_x$ is bounded.} Moreover, the term $\mathcal{L}^{a_i,c_j}(\tilde{\varphi})(x)$ is well-defined, even though $g$ is not twice continuously differentiable everywhere, since $\tilde{\varphi}$ is twice continuously differentiable for any $w<2x$. 
	
	Since $\tau_{\pi}>0$ a.s. and the following results hold
	\begin{equation*}
		\left|\frac{1}{h}\int_0^{\tau_{\pi}\wedge h}e^{-qs}\mathcal{L}^{a_i,c_j}(\tilde{\varphi})(X_s^{\pi})ds\right|\leq \sup_{w\in[0,\infty)}\left|\mathcal{L}^{a_i,c_j}(\tilde{\varphi})(w)\right|,
	\end{equation*}
	and
	\begin{equation*}
		\lim_{h\to 0^+}\frac{1}{h}\int_0^{\tau_{\pi}\wedge h}e^{-qs}\mathcal{L}^{a_i,c_j}(\tilde{\varphi})(X_s^{\pi})ds=\mathcal{L}^{a_i,c_j}(\tilde{\varphi})(x)\mbox{ a.s.,}
	\end{equation*}
	then, via the bounded convergence theorem and inequality \eqref{inequality expectation},
	\begin{equation*}
		\mathcal{L}^{a_i,c_j}(\varphi)(x)=\mathcal{L}^{a_i,c_j}(\tilde{\varphi})(x)\leq 0.
	\end{equation*}
	Thus, $V^{a_i,c_j}$ is a viscosity supersolution at $x$.
	
	(Part II) We now show that $V^{a_i,c_j}$ is a viscosity subsolution of \eqref{hjb main} via contradiction. Suppose otherwise that $V^{a_i,c_j}$ is not a viscosity subsolution. Then there exist $\epsilon>0$, $0<h<\frac{x}{2}$ and a twice continuously differentiable function $\psi$ with $\psi(x)=V^{a_i,c_j}(x)$ such that $\psi\geq V^{a_i,c_j}$,
	\begin{equation}\label{contra sub}
		\max\{\mathcal{L}^{a_i,c_j}(V^{a_i,c_j})(w),V^{a_{i+1},c_{j}}(w)-V^{a_i,c_j}(w),V^{a_i,c_{j+1}}(w)-V^{a_i,c_j}(w)\}\leq -q\epsilon<0
	\end{equation}
	for $w\in[x-h,x+h]$, and
	\begin{equation}\label{contra sub 2}
		V^{a_i,c_j}(w)\leq \psi(w)-\epsilon
	\end{equation}
	for $w\notin[x-h,x+h]$.
	
	Consider the reserve process $\{X^{\pi}_t\}_{t\geq 0}$ corresponding to an admissible strategy $\pi\in\Pi^{\mathscr{A},\mathscr{C}}_{x,a_i,c_j}$. Define the stopping time $\tau^*$ given by
	\begin{equation*}
		\tau^*:=\inf\{t>0:X_t^{\pi}\notin[x-h,x+h]\}.
	\end{equation*}
	Taking expectations and using \eqref{contra sub} yield
	{\begin{equation}\label{contra sub 3}
    \begin{aligned}
		\mathbb{E}\left[e^{-q(\tau_{\pi}\wedge\tau^*)}\psi(X^{\pi}_{\tau^*})\right]-\psi(x)&=\mathbb{E}\left[\int_0^{\tau_{\pi}\wedge\tau^*}e^{-qs}\mathcal{L}^{a_i,c_j}(\psi)(X^{\pi}_s)ds-\int_0^{\tau_{\pi}\wedge\tau^*}e^{-qs}c_j\phantom{.}ds\right]\\
		&\leq \mathbb{E}\left[\int_0^{\tau_{\pi}\wedge\tau^*}e^{-qs}(-q\epsilon)ds+\frac{c_j}{q}\left(e^{-q(\tau_{\pi}\wedge\tau^*)}-1\right)\right]\\
		&=\left(\epsilon+\frac{c_j}{q}\right)\mathbb{E}\left[e^{-q(\tau_{\pi}\wedge\tau^*)}-1\right].
        \end{aligned}
	\end{equation}}
	From \eqref{contra sub 2} and \eqref{contra sub 3},
	{\begin{equation*}
    \begin{aligned}
		\mathbb{E}\left[e^{-q(\tau_{\pi}\wedge\tau^*)}V^{a_i,c_j}(X^{\pi}_{\tau_{\pi}\wedge\tau^*})\right]&\leq \mathbb{E}\left[e^{-q(\tau_{\pi}\wedge\tau^*)}\left(\psi(X^{\pi}_{\tau_{\pi}\wedge\tau^*})-\epsilon\right)\right]\\
		&\leq \psi(x)+\left(\epsilon+\frac{c_j}{q}\right)\mathbb{E}\left[e^{-q(\tau_{\pi}\wedge\tau^*)}-1\right]-\epsilon\mathbb{E}\left[e^{-q(\tau_{\pi}\wedge\tau^*)}\right]\\
		&=\psi(x)+\frac{c_j}{q}\mathbb{E}\left[e^{-q(\tau_{\pi}\wedge\tau^*)}-1\right]-\epsilon.
        \end{aligned}
	\end{equation*}}
	Using Lemma \ref{dpp}, we have
	{\begin{equation*}
		V^{a_i,c_j}(x)=\sup_{\pi\in\Pi^{\mathscr{A},\mathscr{C}}_{x,a_i,c_j}}\mathbb{E}\left[-\frac{c_j}{q}\left(e^{-q(\tau_{\pi}\wedge\tau^*)}-1\right)+e^{-q(\tau_{\pi}\wedge\tau^*)}V^{a_i,c_j}(X^{\pi}_{\tau_{\pi}\wedge\tau^*})\right]\leq \psi(x)-\epsilon,
	\end{equation*}}
	which contradicts the assumption that $V^{a_i,c_j}(x)=\psi(x)$. Therefore, $V^{a_i,c_j}$ is a viscosity subsolution of \eqref{hjb main}.

    The uniqueness result is a direct consequence of Propositions \ref{Lipschitz property} and \ref{comparison viscosity}.
\end{proof}

\begin{proof}[Proof of Theorem \ref{y and z optimal implies verification}]

    By definition, $\wopt(x,a_m,c_n)=V^{a_m,c_n}(x)$. Assume that $\wopt(\cdot,a_k,c_l)=V^{a_k,c_l}$ for $k=i+1,\ldots,m$ and $l=j+1,\ldots,n$. Moreover, assume that $\wopt(\cdot,a_{i+1},c_j)=V^{a_{i+1},c_j}$ and $\wopt(\cdot,a_i,c_{j+1})=V^{a_i,c_{j+1}}$.
	
	By construction, we know that
	\begin{equation*}
		\begin{cases}
			V^{a_{i+1},c_j}(x)-\wopt(x,a_i,c_j)=0,&\mbox{if $x\geq y^*_{i,j}$},\\
			V^{a_{i+1},c_j}(x)-\wopt(x,a_i,c_j)\leq 0,&\mbox{if $x< y^*_{i,j}$},\\
			V^{a_{i},c_{j+1}}(x)-\wopt(x,a_i,c_j)=0,&\mbox{if $x\geq z^*_{i,j}$},\\
			V^{a_{i},c_{j+1}}(x)-\wopt(x,a_i,c_j)\leq 0,&\mbox{if $x< z^*_{i,j}$},\\
			\mathcal{L}^{a_i,c_j}\left(\wopt\right)(x,a_i,c_j)=0,&\mbox{if $x<y^*_{i,j}\wedge z^*_{i,j}$}.
		\end{cases}
	\end{equation*}
    Suppose Theorem \ref{verification discrete} holds. It suffices to show that $\wopt(\cdot,a_i,c_j)$ is a viscosity solution of \eqref{hjb main}. It remains to show the following:
	\begin{equation*}
		\mathcal{L}^{a_i,c_j}\left(\wopt\right)(x,a_i,c_j)\leq 0\quad\mbox{for all $x\geq y^*_{i,j}\wedge z^*_{i,j}$}.
	\end{equation*}
	
	
	Fix $j$. Suppose $x\neq y^*_{k,j}$ for $k=i+1,\ldots,m-1$. Then $x$ belongs to one of the open intervals wherein $\mathcal{L}^{a_k,c_j}(\wopt)(x,a_i,c_j)=0$. Suppose $\mathcal{L}^{a_i,c_j}(\wopt)(x,a_i,c_j)>0$. Then,
	\begin{equation}\label{ineq contradict}
    \begin{aligned}
		0&<\mathcal{L}^{a_i,c_j}(\wopt)(x,a_i,c_j)-\mathcal{L}^{a_k,c_j}(\wopt)(x,a_i,c_j)\\
		&=\frac{1}{2}\sigma^2(a_i^2-a_k^2)\wopt_{xx}(x,a_i,c_j)+\mu(a_i-a_k)\wopt_x(x,a_i,c_j).
        \end{aligned}
	\end{equation}
	
	
	There exist $\delta>0$ and some $k>i$ such that $\mathcal{L}^{a_k,c_j}(\wopt)(x,a_i,c_j)=0$ in $(y^*_{i,j},y^*_{i,j}+\delta)$. Then,
	\begin{equation*}
		\mathcal{L}^{a_k,c_j}(\wopt)(y^{*+}_{i,j},a_i,c_j)=0\quad\mbox{and}\quad\mathcal{L}^{a_i,c_j}(\wopt)(y^{*-}_{i,j},a_i,c_j)=0.
	\end{equation*}
    By Lemma \ref{W is cont diff} and inequality \eqref{ineq contradict},
	\begin{equation*}
    \begin{aligned}
		0&=\mathcal{L}^{a_i,c_j}(\wopt)(y^{*-}_{i,j},a_i,c_j)-\mathcal{L}^{a_k,c_j}(\wopt)(y^{*+}_{i,j},a_i,c_j)\\
		&=\frac{1}{2}\sigma^2\left[a_i^2\wopt_{xx}(y^{*-}_{i,j},a_i,c_j)-a_k^2\wopt_{xx}(y^{*+}_{i,j},a_i,c_j)\right]+\mu(a_i-a_k)\wopt_x(y^{*-}_{i,j},a_i,c_j)\\
		&>\frac{1}{2}\sigma^2\left[a_i^2\wopt_{xx}(y^{*-}_{i,j},a_i,c_j)-a_k^2\wopt_{xx}(y^{*+}_{i,j},a_i,c_j)\right]-\frac{1}{2}\sigma^2(a_i^2-a_k^2)\wopt_{xx}(y^{*-}_{i,j},a_i,c_j)\\
		&=\frac{1}{2}\sigma^2a_k^2\left[\wopt_{xx}(y^{*-}_{i,j},a_i,c_j)-\wopt_{xx}(y^{*+}_{i,j},a_i,c_j)\right].
        \end{aligned}
	\end{equation*}
By Lemma \ref{Wxx inequality}, $\wopt_{xx}(y^{*-}_{i,j},a_i,c_j)-\wopt_{xx}(y^{*+}_{i,j},a_i,c_j)\geq 0$, which is a contradiction to the inequality above. Thus, $\mathcal{L}^{a_i,c_j}(\wopt)(x,a_i,c_j)\leq 0$ for $x\neq y^*_{k,j}$, where $k=i+1,\ldots,m-1$. 
	
	Consider now the case $x=y^*_{k,j}$ with $k=i+1,\ldots,m-1$ and $y^*_{k,j}\geq y^*_{i,j}$. Since $\wopt(x,a_i,c_j)$ may not be twice differentiable at $x=y^*_{k,j}$, we prove that $\mathcal{L}^{a_i,c_j}(\wopt)(x,a_i,c_j)\leq 0$ in the viscosity sense. Take a test function $\varphi_1$ as a supersolution at $y^{*}_{k,j}$. Since $\varphi_1$ is a supersolution, then $\wopt(\cdot,a_i,c_j)-\varphi_1(\cdot)$ achieves its minimum at $y^{*}_{i,j}$. Hence,
	\begin{equation}\label{varphi supersol}
		\varphi_1(y^{*}_{k,j})=\wopt(y^{*}_{k,j},a_i,c_j)\quad\mbox{and}\quad \varphi_1'(y^{*}_{k,j})=\wopt_x(y^{*}_{k,j},a_i,c_j).
	\end{equation}
   Moreover,
	\begin{equation*}
		\wopt_{xx}(y^{*+}_{k,j},a_i,c_j)-\varphi_1''(y^{*}_{k,j})\geq 0\quad\mbox{and}\quad\wopt_{xx}(y^{*-}_{k,j},a_i,c_j)-\varphi_1''(y^{*}_{k,j})\geq0, 
	\end{equation*}
    or, equivalently,
	\begin{equation}\label{varphi 2nd derivative}
		\varphi_1''(y^{*}_{k,j})\leq\min\left\{\wopt_{xx}(y^{*-}_{k,j},a_i,c_j),\wopt_{xx}(y^{*+}_{k,j},a_i,c_j)\right\}.
	\end{equation}
   Using \eqref{varphi supersol} and \eqref{varphi 2nd derivative} yields
	\begin{equation*}
    \begin{aligned}
		\mathcal{L}^{a_i,c_j}(\varphi_1)(y^*_{k,j})&=\frac{1}{2}\sigma^2a_i^2\varphi_1''(y^*_{k,j})+(\mu a_i-b-c_j)\varphi_1'(y^*_{k,j})-q\varphi_1(y^*_{k,j})+c_j\\
		&=\frac{1}{2}\sigma^2a_i^2\varphi_1''(y^*_{k,j})+(\mu a_i-b-c_j)\wopt_x(y^*_{k,j},a_i,c_j)\\
  &\quad-q\wopt(y^*_{k,j},a_i,c_j)+c_j\\
		&\leq\min\left\{\mathcal{L}^{a_i,c_j}(\wopt)(y^{*-}_{k,j},a_i,c_j),\mathcal{L}^{a_i,c_j}(\wopt)(y^{*+}_{k,j},a_i,c_j)\right\}\\
		&=0.
        \end{aligned}
	\end{equation*}
    The proof for the dividend threshold levels is similar.
\end{proof}

\section{Value Function Derivation Using Scale Functions}\label{appendix:c}

    {We present an alternative derivation of the value function using scale functions. 
    Write $X^{i,j}_t=x+(\mu a_i-b-c_j)t+\sigma a_i W_t$ for $i=1,\ldots, m$ and $j=1,\ldots,n$. For a fixed $\beta \geq 0$, we define the following passage times:
    \begin{equation*}
        \tau_{\beta}^{i,j}:=\inf \left\{t\geq 0: X^{i,j}_t \geq \beta\right\},\quad 
        \tau_{0}^{i,j}:=\inf \left\{t\geq 0: X^{i,j}_t <0\right\}.
    \end{equation*}
    We also define the functions $\mathbb W_{i,j}^{q}(x)$ and $\mathbb Z_{i,j}^{q}(x)$ of the process $X^{i,j}:=\{X^{i,j}_t\}_{t\geq 0}$ as
    \begin{equation*}
        \begin{aligned}
            \int_0^{\infty} e^{-ux}\,\mathbb W_{i,j}^{(q)}(x) \, dx&=\frac{1}{\psi_{i,j}(u)-q},\quad u>\theta_1(a_i,c_j),\\
            \mathbb Z_{i,j}^{(q)}(x)&=1+q\int_0^x \mathbb W_{i,j}^{(q)}(y) \,dy,
        \end{aligned}
    \end{equation*}
    where $\psi_{i,j}(u):=\frac{1}{2}\sigma^2 a_i^2u^2+(\mu a_i-b-c_j)u$ is the Laplace exponent of $X^{i,j}$. The functions $\mathbb W^{q}_{i,j}$ and $\mathbb Z^{q}_{i,j}$ are referred to as the \emph{$q$-scale functions} of $X^{i,j}$ in the literature of exit problems for spectrally negative L\'{e}vy processes. 

    Write $\kappa_{i,j}:=\left((\mu a_i-b-c_j)^2+2q\sigma^2a_i^2\right)^{-1/2}$. Using the method of partial fractions and the Laplace inverse transform of $(\psi_{i,j}(u)-q)^{-1}$, we obtain
    \begin{equation*}
        \mathbb W^{(q)}_{i,j}(x)=\kappa_{i,j}\left(e^{\theta_1(a_i,c_j)x}-e^{\theta_2(a_i,c_j)x}\right).
    \end{equation*}
    Consequently, it can be shown that
    \begin{equation*}
        \mathbb Z^{(q)}_{i,j}(x)=\frac{q}{\theta_1(a_i,c_j)}\mathbb W^{(q)}_{i,j}(x)+e^{\theta_2(a_i,c_j)x}.
    \end{equation*}

    Suppose $y_{i,j}<z_{i,j}$. For $x\in [0,y_{i,j})$, we have
    \begin{equation*}
        \begin{aligned}
            \upsilon^{i,j}(x)
            &=\mathbb E\left[\int_0^{\tau_{y_{i,j}}^{i,j}\wedge \tau_{0}^{i,j}}e^{-qt}\, c_j\, dt\right]+\mathbb E\left[e^{-q\tau_{y_{i,j}}^{i,j}}\cdot\mathbf{1}_{\{\tau_{y_{i,j}}^{i,j}< \tau_{0}^{i,j}\}}\right]\upsilon^{i+1,j}(x)\\
            &=\frac{c_j}{q}\left[1-\mathbb E\left[e^{-q\left(\tau_{y_{i,j}}^{i,j}\wedge \tau_{0}^{i,j}\right)}\right]\right]+\mathbb E\left[e^{-q\tau_{y_{i,j}}^{i,j}}\cdot\mathbf{1}_{\{\tau_{y_{i,j}}^{i,j}< \tau_{0}^{i,j}\}}\right]\upsilon^{i+1,j}(x)\\
            &=\frac{c_j}{q}\left[1-\mathbb E\left[e^{-q\tau_{y_{i,j}}^{i,j}}\cdot\mathbf{1}_{\{\tau_{y_{i,j}}^{i,j}< \tau_{0}^{i,j}\}}\right]-\mathbb E\left[e^{-q\tau_{0}^{i,j}}\cdot\mathbf{1}_{\{\tau_{y_{i,j}}^{i,j}> \tau_{0}^{i,j}\}}\right]\right]+\mathbb E\left[e^{-q\tau_{y_{i,j}}^{i,j}}\cdot\mathbf{1}_{\{\tau_{y_{i,j}}^{i,j}< \tau_{0}^{i,j}\}}\right]\upsilon^{i+1,j}(x).
        \end{aligned}
    \end{equation*}
    By \citet[Theorem 8.1(iii)]{kyprianou2014book}, we have 
    \begin{equation*}
    \begin{aligned}
        \mathbb E\left[e^{-q\tau_{y_{i,j}}^{i,j}}\cdot\mathbf{1}_{\{\tau_{y_{i,j}}^{i,j}< \tau_{0}^{i,j}\}}\right]&=\frac{\mathbb W_{i,j}^{q}(x)}{\mathbb W_{i,j}^{q}(y_{i,j})}=\frac{e^{\theta_1(a_i,c_j)x}-e^{\theta_2(a_i,c_j)x}}{e^{\theta_1(a_i,c_j)y_{i,j}}-e^{\theta_2(a_i,c_j)y_{i,j}}},\\
        \mathbb E\left[e^{-q\tau_{0}^{i,j}}\cdot\mathbf{1}_{\{\tau_{y_{i,j}}^{i,j}> \tau_{0}^{i,j}\}}\right]&=\mathbb Z^{(q)}_{i,j}(x)- \mathbb Z^{(q)}_{i,j}(y_{i,j})\frac{\mathbb W_{i,j}^{q}(x)}{\mathbb W_{i,j}^{q}(y_{i,j})}=e^{\theta_2(a_i,c_j)x}+(1-e^{\theta_2(a_i,c_j)y_{i,j}})\frac{\mathbb W_{i,j}^{q}(x)}{\mathbb W_{i,j}^{q}(y_{i,j})}.
    \end{aligned}
    \end{equation*}
    Then,
    \begin{equation*}
        \begin{aligned}
            \upsilon^{i,j}(x)
            &=\frac{c_j}{q}\left[1-e^{\theta_2(a_i,c_j)x}-(1-e^{\theta_2(a_i,c_j)y_{i,j}})\frac{\mathbb W_{i,j}^{q}(x)}{\mathbb W_{i,j}^{q}(y_{i,j})}\right]+\frac{\mathbb W_{i,j}^{q}(x)}{\mathbb W_{i,j}^{q}(y_{i,j})}\upsilon^{i+1,j}(x)\\
            &=\frac{c_j}{q}\left(1-e^{\theta_2(a_i,c_j)x}\right)+\left(e^{\theta_1(a_i,c_j)x}-e^{\theta_2(a_i,c_j)x}\right)\frac{\upsilon^{i+1,j}(x)-\frac{c_j}{q}(1-e^{\theta_2(a_i,c_j)y_{i,j}})}{e^{\theta_1(a_i,c_j)y_{i,j}}-e^{\theta_2(a_i,c_j)y_{i,j}}},
        \end{aligned}
    \end{equation*}
    which yields the form of the value function in Theorem \ref{y and z optimal implies verification} for $i<m$ or $j<n$. Moreover, by construction, if $x>y_{i,j}$, then $\upsilon^{i,j}(x)=\upsilon^{i+1,j}(x)$. Similar arguments apply for the cases $x<z_{i,j}<y_{i,j}$ and $x<y_{i,j}=z_{i,j}$. 

    For the case where $i=m$ and $j=n$, we have $y_{m,n}=z_{m,n}=\infty$. Moreover,
    \begin{equation*}
        \lim_{y\to\infty}\frac{\mathbb W_{i,j}^{q}(x)}{\mathbb W_{i,j}^{q}(y)}=\frac{q}{\theta_1(a_m,c_n)}.
    \end{equation*}
    Hence, we obtain
    \begin{equation*}
        \upsilon^{m,n}(x)= \mathbb E\left[\int_0^{\tau_{0}^{m,n}}e^{-qt}\, c_n\, dt\right]=\frac{c_n}{q}\left[1-\mathbb E\left[e^{-q\tau_{0}^{m,n}}\right]\right]=\frac{c_n}{q}\left(1-e^{\theta_2(a_m,c_n)x}\right),
    \end{equation*}
    which completes the form of the value function in Theorem \ref{y and z optimal implies verification}.}

\bibliographystyle{plainnat} 
\bibliography{UpdatedReferences}

\end{document}